\documentclass[aop]{imsart}

\RequirePackage[OT1]{fontenc}
\RequirePackage{amsthm,amsmath}
\RequirePackage[numbers]{natbib}
\RequirePackage[colorlinks,citecolor=blue,urlcolor=blue]{hyperref}

\usepackage{amssymb}
\usepackage{amsfonts}
\usepackage{graphicx}
\usepackage{bbm}
\usepackage{color}

\startlocaldefs
\newtheorem{thm}{Theorem}[section]
\newtheorem{corollary}[thm]{Corollary}

\newtheorem{assumption}{Assumption}
\newtheorem{lemma}[thm]{Lemma}

\newtheorem{prop}[thm]{Proposition}

\newtheorem{theorem}[thm]{Theorem}
\theoremstyle{definition}

\newtheorem{rem}[thm]{Remark}

\numberwithin{equation}{section}
\newtheorem*{acknowledgement}{Acknowledgement}
\endlocaldefs

\begin{document}

\begin{frontmatter}

\title{Coupling in the Heisenberg group and its applications to gradient estimates}
\runtitle{Heisenberg coupling}

\author{\fnms{Sayan} \snm{Banerjee}\ead[label=Sayan]{sayan@email.unc.edu}\thanksref{t1}},
\author{\fnms{Maria} \snm{Gordina} \corref{Masha} \ead[label=Masha]{maria.gordina@uconn.edu} \thanksref{t2}}
\and
\author{\fnms{Phanuel} \snm{Mariano} \ead[label=Phanuel]{phanuel.mariano@uconn.edu} \thanksref{t3}}

\thankstext{t1}{Research was supported in part by EPSRC Research Grant  EP/K013939}
\thankstext{t2}{Research was supported in part by the Simons Fellowship and NSF Grant DMS-1007496}
\thankstext{t3}{Research was supported in part by NSF Grant DMS-1007496}

\address{
Department of Statistics and Operations Research\\
University of North Carolina\\
Chapel Hill, NC 27599, U.S.A.\\
\printead{Sayan}}

\address{
Department of Mathematics\\
University of Connecticut\\
Storrs, CT 06269,  U.S.A.\\
\printead{Masha}}

\address{
Department of Mathematics\\
University of Connecticut\\
Storrs, CT 06269,  U.S.A.\\
\printead{Phanuel}}

\runauthor{Banerjee, Gordina, Mariano}

\begin{abstract}
We construct a non-Markovian coupling for hypoelliptic diffusions which are Brownian motions in the three-dimensional Heisenberg group. We then derive properties of this coupling such as  estimates on the coupling rate, and upper and lower bounds on the total variation distance between the laws of the Brownian motions. Finally we use these properties to prove gradient estimates for harmonic functions for the hypoelliptic Laplacian which is the generator of Brownian motion in the Heisenberg group.

\end{abstract}



\end{frontmatter}

\section{Introduction}\label{s.1}

Recall that a coupling of two probability measures $\mu_1$ and $\mu_2$, defined on respective measure spaces $(\Omega_1, \mathcal{A}_1)$ and $(\Omega_2, \mathcal{A}_2)$, is a measure $\mu$ on the product space $(\Omega_1 \times \Omega_2, \mathcal{A}_1 \times \mathcal{A}_2)$ with marginals $\mu_1$ and $\mu_2$. In this article, we will be interested in coupling of the laws of two Markov processes $(X_t: t \geqslant 0)$ and $(Y_t: t\geqslant 0)$ in a geometric setting of a sub-Riemannian manifold such as the Heisenberg group $\mathbb{H}^{3}$. Namely, we discuss couplings of two Markov processes having the same generator but starting from different points joining together (coupling) at some random time, and how these can be used to obtain total variation bounds and prove gradient estimates for harmonic functions on $\mathbb{H}^{3}$.
Couplings have been an extremely useful tool in probability theory and has resulted in establishing deep connections between probability, analysis and geometry.

We start by providing some background on couplings and then on gradient estimates in our setting. The coupling is said to be \textit{successful} if the two processes couple within finite time almost surely, that is, the \textit{coupling time} for $X_{t}$ and $Y_{t}$ defined as

\[
\tau(X,Y)=\inf\{t \geqslant 0: X_s = Y_s \text{ for all } s \geqslant t\}.
\]
is almost surely finite.

A major application of couplings arises in estimating the \textit{total variation distance} between the laws of two Markov processes at time $t$ which in general is very hard to compute explicitly. Such an estimate can be obtained from the \textit{Aldous' inequality}

\begin{equation}\label{equation:Aldous}
\mu\left\{\tau(X,Y) >t\right\} \geqslant ||\mathcal{L}(X_t)-\mathcal{L}(Y_t)||_{TV},
\end{equation}
where $\mu$ is the coupling of the Markov processes $X$ and $Y$, $\mathcal{L}(X_t)$ and $\mathcal{L}(Y_t)$ denote the laws (distributions) of $X_t$ and $Y_t$ respectively, and
\[
||\nu||_{TV} = \sup\{|\nu(A)|: A \text{ measurable}\}
\]
denotes the total variation norm of the measure $\nu$.

This, in turn, can be used to provide sharp rates of convergence of Markov processes to their respective stationary distributions, when they exist (see \cite{LevinPeresWilmerBook} for some such applications in studying mixing times of Markov chains).

This raises a natural question: how can we couple two Markov processes so that the probability of failing to couple by time $t$ (coupling rate) is minimized (in an appropriate sense) for some, preferably all, $t$? Griffeath \cite{Griffeath1975} was the first to prove that \textit{maximal couplings}, that is, the couplings for which the Aldous' inequality becomes an equality for each $t$ in the time set of the Markov process, exist for discrete time Markov chains. This was later greatly simplified by Pitman \cite{Pitman1976a} and generalized to non-Markovian processes by Goldstein \cite{Goldstein1979} and continuous time c\`{a}dl\`{a}g processes by Sverchkov and Smirnov \cite{SverchkovSmirnov1990}.

These constructions, though extremely elegant, have a major drawback: they are typically very implicit. Thus, it is very hard, if not impossible, to perform detailed calculations and obtain precise estimates using these couplings. Part of the implicitness comes from the fact that these couplings are non-Markovian.

A Markovian coupling of two Markov processes $X$ and $Y$ is a coupling where, for any $t \geqslant 0$, the joint process $\{(X_s,Y_s): s \geqslant t\}$ conditioned on the filtration $\sigma\{(X_s, Y_s): s \leqslant t\}$ is again a coupling of the laws of $X$ and $Y$, but now starting from $(X_t,Y_t)$. These are the most widely used couplings in deriving estimates and performing detailed calculations as their constructions are typically explicit. However, these couplings usually do not attain the optimal rates. In fact, it has been shown in \cite{BanerjeeKendall2017a} that the existence of a maximal coupling that is also Markovian imposes enormous constraints on the generator of the Markov process and its state space. Further, \cite{BanerjeeKendall2016a} describes an example using Kolmogorov diffusions defined as a two dimensional diffusion given by a standard Brownian motion along with its running time integral, where for \textit{any Markovian coupling}, the probability of failing to couple by time $t$ does not even attain the same order of decay (with $t$) as the total variation distance. More precisely, they showed that if the driving Brownian motions start from the same point, then the total variation distance between the corresponding Kolmogorov diffusions decays like $t^{-3/2}$ whereas for any Markovian coupling, the coupling rate is at best of order $t^{-1/2}$.

This brings us to the main subject of this article: when can we produce non-Markovian couplings that are explicit enough to give us good bounds on the total variation distance between the laws of $X_t$ and $Y_t$ when Markovian couplings fail to do so? And what information can such couplings provide about the geometry of the state space of these Markov processes? In this article, we look at the Heisenberg group which is the simplest example of a sub-Riemannian manifold and Brownian motion on it. The latter is the Markov process whose generator is the sub-Laplacian on the Heisenberg group as described in Section \ref{section:prelim}. We construct an explicit successful non-Markovian coupling of two copies of this process starting from different points in $\mathbb{H}^{3}$ and use it to derive sharp bounds on the total variation distance between their laws at time $t$. We also use this coupling to produce gradient estimates for harmonic functions on the Heisenberg group (more details below), thus providing a non-trivial link between probability and geometric analysis in the sub-Riemannian setting.

We note here that successful \textit{Markovian couplings} of Brownian motions on the Heisenberg group have been constructed in \cite{Kendall2007a} and rates of these couplings have been studied in \cite{Kendall2010a}. However, the rates for the coupling we construct are much better. In fact, we show in Remark \ref{rem:Mark} that it is impossible to derive the rates we get from Markovian couplings. Moreover, the coupling we consider is efficient, that is, the coupling rate and the total variation distance decay like the same power of $t$ as pointed out in Remark \ref{rem:efficiency}.

Now we would like to describe gradient estimates in geometric settings and how couplings have been used to prove them previously. Let us start with a classical gradient estimate for harmonic functions in $\mathbb{R}^{d}$. Suppose $u$ is a real-valued function  $u$ on $\mathbb{R}^{d}$ which is harmonic in a ball $B_{2\delta}(x_0)$, then there exists a positive constant $C_d$ (which depends only on the dimension $d$ and not on $u$) such that
\[
\sup_{x\in B_{\delta}(x_{0})}\left|\nabla u(x)\right|\leqslant \frac{C_d}{\delta}\sup_{x\in B_{2\delta}(x_{0})}\left|u(x)\right|.
\]
In 1975, Cheng and Yau (see \cite{ChengYau1975a, Yau1975a, SchoenYau}) generalized the classical gradient estimate to complete Riemannian manifolds $M$ of dimension $d\geqslant 2$ with Ricci curvature bounded below by $-(d-1)K$ for some $K\geqslant 0$. They proved that any
positive harmonic function on a Riemannian ball $B_{\delta}(x_0)$ satisfies
\[
\sup_{x \in B_{\delta/2}(x_0)}\frac{\left|\nabla u(x)\right|}{u(x)}\leqslant  C_{d}\left(\frac{1}{\delta}+\sqrt{K}\right).
\]
Moreover, in addition to such estimates, there is a vast literature on functional inequalities such as heat kernel gradient estimates, Poincar\'e inequalities, heat kernel estimates, elliptic and parabolic Harnack inequalities etc on Riemannian manifolds or more generally on measure metric spaces. Quite often these results require assumptions such as volume doubling and curvature bounds.

In 1991, M.~Cranston in \cite{Cranston1991} used the method of coupling two diffusion processes to obtain a similar gradient estimate for solutions to the equation

\begin{equation}\label{CranstonEQ}
\frac{1}{2}\Delta u+Zu=0
\end{equation}
on a Riemannian manifold $\left(M,g\right)$ whose Ricci curvature is bounded below and $Z$ is a bounded vector field. This coupling is known as the Kendall-Cranston coupling  as it was based on the techniques in \cite{Kendall1989a}. In particular, M.~Cranston proved the following gradient estimate.

\begin{theorem}[Cranston] Suppose $\left(M,g\right)$ is a complete d-dimensional
Riemannian manifold with distance $\rho_{M}$ and assume $\mbox{Ric}_{M}\geqslant -Kg.$
Let $Z$ be a $C^{1}$ vector field on $M$ such that $\left|Z(x)\right|\leqslant m$
for all $x\in M$. There is a constant $c=c\left(K, d, m\right)$ such
that whenever $\delta>0$ and \eqref{CranstonEQ} is satisfied in some Riemannian ball $B_{2\delta}\left(x_{0}\right)$,
we have
\[
\left|\nabla u(x)\right|\leqslant c\left(\frac{1}{\delta}+1\right)\sup_{x \in B\left(x_{0},3\delta/2\right)}|u(x)|,\,\,\,\,\, x\in B\left(x_{0},\delta\right).
\]
If \eqref{CranstonEQ} is satisfied on $M$ and $u$ is bounded and positive, then
\[
\left|\nabla u(x)\right|\leq2\left(\sqrt{K\left(d-1\right)}+m\right)\left\Vert u\right\Vert _{\infty}.
\]

\end{theorem}
Cranston's approach generalized the coupling of Brownian motions on manifolds of Kendall \cite{Kendall1986} to couple processes with the generator $L=\frac{1}{2}\Delta +Z$. The methods in that paper required tools from Riemannian geometry such as the Laplacian comparison theorem and the index theorem to obtain estimates on the processes $\rho_{M}\left(X_t,Y_t\right)$ and $\rho_{M}\left(X_t,X_0\right)$ where $\rho_M$ is the Riemannian distance. M.~Cranston also proved similar results on $\mathbb{R}^d$ in \cite{Cranston1992a}.

In this paper we consider the simplest sub-Riemannian manifold, the Heisenberg group $\mathbb{H}^{3}$ as a starting point of using couplings for proving gradient estimates in such a setting. As the generator of $\mathbb{H}^{3}$-valued Brownian motion is a hypoelliptic operator, functional inequalities for the corresponding harmonic functions or hypo-elliptic heat kernels are much more challenging to prove. There was recent progress in using generalized curvature-dimension inequalities for such results (e.g. \cite{BakryBaudoinBonnefontChafai2008, BaudoinGarofalo2017, BaudoinBonnefontGarofalo2014}, as well as results in the spirit of optimal transport (e.g. \cite{Kuwada2010a}). The main point of the current paper is not whether a coupling can be constructed, as these have been known since \cite{BenArousCranstonKendall1995}, but rather finding a (necessarily non-Markovian) coupling that gives sharp total variation bounds and explicit gradient estimates. The properties of the coupling we construct in the current paper are crucial in this, and it is interesting to contrast this with optimality (or the lack of it) for the Kendall-Cranston coupling in the Riemannian manifolds as described in \cite{KuwadaSturm2007, Kuwada2007}.

The paper is organized as follows. Section \ref{section:prelim} gives basics on sub-Riemannian manifolds and the Heisenberg group $\mathbb{H}^{3}$ including Brownian motion on $\mathbb{H}^{3}$. In Section \ref{section:3} we construct the non-Markovian coupling of Brownian motions in $\mathbb{H}^{3}$, and describe its properties. Finally, in Section \ref{section:Gradient estimates} we prove the gradient estimates for harmonic functions for the hypoelliptic Laplacian which is the generator of Brownian motion in the Heisenberg group.

\section{Preliminaries}\label{section:prelim}

\subsection{Sub-Riemannian basics}  A sub-Riemannian manifold $M$ can be thought of as a Riemannian manifold where we have a constrained movement. Namely, such a manifold has the structure $\left(M,\mathcal{H}, \langle \cdot, \cdot \rangle \right)$, where allowed directions are only the ones in the horizontal distribution, which is a suitable subbundle $\mathcal{H}$ of the tangent bundle $TM$. For more detail on sub-Riemannian manifolds we refer to \cite{MontgomeryBook2002}.

Namely, for a smooth connected $d$-dimensional manifold $M$ with the tangent bundle $TM$, let  $\mathcal{H} \subset TM$ be an $m$-dimensional smooth sub-bundle such that the sections of $\mathcal{H}$ satisfy  \emph{H\"{o}rmander's condition} (the \emph{bracket generating condition}) formulated in Assumption \ref{HCassump}. We assume that on each fiber of $\mathcal{H}$ there is an inner product $\langle \cdot, \cdot \rangle$ which varies smoothly between fibers.  In this case, the triple $(M, \mathcal{H}, \langle \cdot, \cdot \rangle)$ is called a \emph{sub-Riemannian manifold} of rank $m$, $\mathcal{H}$ is called the \emph{horizontal distribution}, and $\langle \cdot, \cdot \rangle$ is called the \emph{sub-Riemannian metric}. The vectors (resp.~vector fields)  $X \in \mathcal{H}$ are called \emph{horizontal vectors} (resp.~horizontal vector fields), and curves $\gamma$ in $M$ whose tangent vectors are horizontal, are called \emph{horizontal curves}.

\begin{assumption}(H\"{o}rmander's condition)\label{HCassump}
We will say  that $\mathcal{H}$ satisfies \emph{H\"{o}rmander's (bracket generating) condition} if horizontal vector fields with their Lie brackets span the tangent space $T_{p}M$ at every point $p \in M$.
\end{assumption}
H\"{o}rmander's condition guarantees analytic and topological properties such as hypoellipticity of the corresponding sub-Laplacian and topological properties of the sub-Riemannian manifold $M$. We explain briefly both aspects below. First we define the \emph{Carnot-Carath\'{e}odory metric} $d_{CC}$ on $M$ by
\begin{align}\label{e.3.1}
& d_{CC}(x,y) =
\\
& \inf\left\{  \int_0^1 \left\Vert \gamma^{\prime}(t)\right\Vert _{\mathcal{H}} \, dt~  \text{ where } \gamma(0) = x, \gamma(1) = y, \gamma \text{ is a horizontal curve}\right\}, \notag
\end{align}
where as usual $\inf(\emptyset):= \infty$. Here the norm is induced by the inner product on $\mathcal{H}$, namely, $\left\Vert v\right\Vert _{\mathcal{H}} :=\left(\langle v, v \rangle_p \right)^{\frac{1}{2}}$ for $v \in \mathcal{H}_{p}, \ p \in M$. The Chow-Rashevski theorem says that H\"{o}rmander's condition is sufficient to ensure that any two points in $M$ can be connected by a finite length horizontal curve. Moreover, the topology generated by the the Carnot-Carath\'{e}odory metric coincides with the original topology of the manifold $M$.

As we are interested in a Brownian motion on a sub-Riemannian manifold $(M, \mathcal{H}, \langle \cdot, \cdot \rangle)$, a natural question is what its generator is. While there is no canonical operator such as the Laplace-Beltrami operator on a Riemannian manifold, there is a notion of a \emph{sub-Laplacian} on sub-Riemannian manifolds. A second order differential operator defined on $C^{\infty}\left( M \right)$ is called a \emph{sub-Laplacian} $\Delta_{\mathcal{H}}$ if for every $p \in M$ there is a neighborhood $U$ of $p$ and a collection of smooth vector fields $\{X_0, X_1, ..., X_m\}$ defined on $U$  such that $\{X_1, ..., X_m\}$ are orthonormal with respect to the sub-Riemannian metric and
\[
\Delta_{\mathcal{H}} = \sum\limits_{k=1}^m X_k^2 + X_0.
\]
By the classical theorem of L.~H\"{o}rmander in \cite[Theorem 1.1]{Hormander1967a} H\"{o}rmander's condition (Assumption \ref{HCassump}) guarantees that any sub-Laplacian is hypoelliptic. For more properties of sub-Laplacians which are generators of a Brownian motion on a sub-Riemannian manifold  we refer to \cite{GordinaLaetsch2016a}.

Finally, the \emph{horizontal gradient} $\nabla_{\mathcal{H}}$ is a horizontal vector field such that for any smooth $f:M\to\mathbb{R}$ we have that for all $X\in\mathcal{H}$,
\[
\langle\nabla_{\mathcal{H}}f, X\rangle=X\left(f\right).
\]
We define the length of the gradient as in \cite{Kuwada2010a}. For a function $f$ on $M$, let
\begin{equation}\label{Grad}
\left|\nabla_{\mathcal{H}}f\right|\left(x\right):= \lim_{r\downarrow0}\sup_{0<d_{CC}\left(x,\tilde{x}\right)\leqslant  r}\left|\frac{f\left(x\right)-f\left(\tilde{x}\right)}{d_{CC}\left(x,\tilde{x}\right)}\right|,
\end{equation}
and set $\left\Vert \nabla_{\mathcal{H}}f\right\Vert _{\infty}:=\sup_{x\in\mathbb{H}^{3}}\left|\nabla_{\mathcal{H}}f\right|\left(x\right)$.

\subsection{The Heisenberg group}

The Heisenberg group $\mathbb{H}^3$ is the simplest non-trivial example of a sub-Riemannian manifold. Namely, let $\mathbb{H}^3\cong\mathbb{R}^3$ with the multiplication defined by
\[
\left(x_{1}, y_{1},z_{1}\right)\star\left(x_{2},y_{2},z_{2}\right):=\left(x_{1}+x_{2}, y_{1}+y_{2}, z_{1}+z_{2}+\left(x_{1}y_{2}-x_{2}y_{1}\right)\right),
\]
with the group identity $e=\left( 0, 0, 0 \right)$ and the inverse given by $\left(x, y,z\right)^{-1}=\left( -x, -y, -z\right)$.

We define $\mathcal{X}$, $\mathcal{Y}$, and $\mathcal{Z}$ as the unique left-invariant vector fields with $\mathcal{X}_e = \partial_x$, $\mathcal{Y}_e = \partial_y$, and $\mathcal{Z}_e = \partial_z$, so that
\begin{align*}
&\mathcal{X} = \partial_x - y \partial_z,\\
&\mathcal{Y} = \partial_y +  x \partial_z,\\
&\mathcal{Z} = \partial_z.
 \label{eqn.XYheis}
\end{align*}
The horizontal distribution is defined by $\mathcal{H} = \operatorname{span}\{\mathcal{X},\mathcal{Y}\}$ fiberwise. Observe that $[\mathcal{X},\mathcal{Y}] = 2\mathcal{Z}$, so H\"{o}rmander's condition is easily satisfied. Moreover, as any iterated Lie bracket of length greater than two vanishes, $\mathbb{H}^3$ is a nilpotent group of step $2$. The Lebesgue measure on $\mathbb{R}^3$ is a Haar measure on $\mathbb{H}^3$.
We endow $\mathbb{H}^3$ with the sub-Riemannian metric $\langle \cdot, \cdot \rangle$ so that $\{\mathcal{X},\mathcal{Y}\}$ is an orthonormal frame for the horizontal distribution. As pointed out in \cite[Example 6.1]{GordinaLaetsch2016a}, the (sum of squares) operator
\begin{equation}\label{e.SubLaplacian}
\Delta_{\mathcal{H}}=\mathcal{X}^{2}+\mathcal{Y}^{2}
\end{equation}
is a  natural sub-Laplacian for the Heisenberg group with this sub-Riemannian structure.

In general it is very cumbersome to compute the Carnot-Carath\'{e}odory distance $d_{CC}$ explicitly. In the case of the Heisenberg group an explicit formula for the distance is known. Let $r\left(\mathbf{x}\right)=d_{CC}\left(\mathbf{x}, e\right)$
be the distance between $\mathbf{x}=\left(x, y, z \right) \in \mathbb{H}^{3}$ and the identity $e=\left( 0, 0, 0 \right)$. In \cite{ChangTieWu2010} the distance is given by the formula

\[
r\left(\mathbf{x}\right)^2=\nu\left(\theta_{c}\right)\left( x^2+y^2+\left|z\right|\right),
\]
where $\theta_{c}$ is the unique solution of $\mu\left(\theta\right)\left(x^2+y^2\right)=\left|z\right|$ in the interval $[0,\pi)$ and $\mu(z)=\frac{z}{\sin^{2}z}-\cot z$ and where
\[
\nu(z)=\frac{z^{2}}{\sin^{2}z}\frac{1}{1+\mu(z)}=\frac{z^{2}}{z+\sin^{2}z-\sin z\cos z}, \ \nu(0)=2.
\]
Since the distance is left-invariant, we have
\[
d_{CC}\left(\mathbf{x},\tilde{\mathbf{x}}\right)=d_{CC}\left(\tilde{\mathbf{x}}^{-1}\star\mathbf{x},e\right)
\]
which gives us an explicit expression for $d_{CC}$ on the Heisenberg
group. Although $\nu$ is not continuous it was shown in \cite{CalinOvidiuChang2007} that $d_{CC}$ is continuous.

We will not use this explicit expression for $d_{CC}$. Instead, since $\nu\geqslant 0$ and bounded below and above by positive constants
in the interval $[0,\pi)$, it is clear that the Carnot-Carath\'{e}odory distance is equivalent to the pseudo-metric

\begin{equation}\label{e.2.1}
 \rho\left(\mathbf{x}, \mathbf{y}\right)=\left(\left(x-\tilde{x}\right)^{2}+\left(y-\tilde{y}\right)^{2}+\left|z-\tilde{z}+x\tilde{y}-y\tilde{x}\right|\right)^{\frac{1}{2}}.
\end{equation}

Finally, we can describe Brownian motion whose generator is $\Delta_{\mathcal{H}}/2$ explicitly as follows. Let $B_{1}, B_{2}$ be real-valued independent Brownian motions starting from $0$. Define Brownian motion on the Heisenberg group $\mathbf{X}_{t}:\left[0,\infty\right)\times\Omega\to\mathbb{H}$
to be the solution of the following Stratonovich stochastic differential equation (SDE)
\begin{eqnarray*}
d\mathbf{X}_{t} & = & \mathcal{X}\left(\mathbf{X}_{t}\right)\circ dB_{1}(t)+\mathcal{Y}\left(\mathbf{X}_{t}\right)\circ dB_{2}(t),
\\
\mathbf{X}_{0} & = & \left(b_{1},b_{2},a\right).
\end{eqnarray*}

Letting $\mathbf{X}_{t}=\left(X_{1}(t),X_{2}(t),X_{3}(t)\right)$
we see that the SDE reduces to
\[
d\mathbf{X}_{t}=\left(\begin{array}{c}
1\\
0\\
-X_2(t)
\end{array}\right)\circ dB_{1}(t)+\left(\begin{array}{c}
0\\
1\\
X_1(t)
\end{array}\right)\circ dB_{2}(t),
\]
so that one needs to solve the following system of equations
\begin{eqnarray*}
dX_{1}(t) & = & dB_{1}(t)\\
dX_{2}(t) & = & dB_{2}(t),\\
dX_{3}(t) & = & -X_{2}(t)\circ dB_{1}(t)+X_{1}(t)\circ dB_{2}(t).
\end{eqnarray*}
Since the covariation of two independent Brownian motions is zero we get that
\begin{eqnarray}\label{Heiseneq}
X_{1}(t) & = & b_{1}+B_{1}(t),\nonumber\\
X_{2}(t) & = & b_{2}+B_{2}(t),\nonumber\\
X_{3}(t) & = & a+\int_{0}^{t}(B_{1}(s) + b_{1})dB_{2}(s)-\int_{0}^{t}(B_{2}(s) + b_{2})dB_{1}(s).
\end{eqnarray}

\section{Coupling results}\label{section:3}

Let $B_{1},B_{2}$ be independent real-valued Brownian motions, starting from $b_1$ and $b_2$ respectively. We call the process
\begin{equation}\label{e.BMHeisenberg}
\mathbf{X}_{t}=\left(B_{1}(t), B_{2}(t), a+\int_{0}^{t}B_{1}(s)dB_{2}(s)-\int_{0}^{t}B_{2}(s)dB_{1}(s)\right)
\end{equation}
Brownian motion on the Heisenberg group, with driving Brownian
motion $\mathbf{B}=\left(B_{1},B_{2}\right)$, starting from $\left(b_{1}, b_{2}, a\right)$.
Let $\mathbf{X}$ and $\widetilde{\mathbf{X}}$ be coupled copies of this
process starting from $\left(b_{1}, b_{2}, a\right)$ and $\left(\widetilde{b}_{1}, \widetilde{b}_{2},\widetilde{a}\right)$
respectively.  Denote the coupling time
\[
\tau=\inf\left\{ t\geqslant 0:\mathbf{X}_{s}=\widetilde{\mathbf{X}}_{s} \text{ for all } s \geqslant  t\right\}.
\]
%For the coupling construction in this paper, it can be readily checked that the above definition agrees with the definition of coupling time given in the introduction.

We will construct a non-Markovian coupling $\left(\mathbf{X},\widetilde{\mathbf{X}}\right)$
of two Brownian motions on the Heisenberg group. This, via the Aldous' inequality, will yield an upper bound on the total variation distance between the laws of $\mathbf{X}$ and $\widetilde{\mathbf{X}}$. Before we state and prove the main theorem, we describe the tools required
in its proof.

For $T>0$, let $\left(B^{\text{br}},\widetilde{B}^{\text{br}}\right)$ be a coupling of standard Brownian bridges defined on the interval $\left[0, T\right]$. If $G^{(T)}$ is a Gaussian variable with mean zero and variance $T$ independent of $\left(B^{\text{br}}, \widetilde{B}^{\text{br}}\right)$,
a standard covariance computation shows that the assignment
\begin{eqnarray}
B(t) & = & B^{\text{br}}(t)+\frac{t}{T}G^{(T)}\nonumber \\
\widetilde{B}(t) & = & \widetilde{B}^{\text{br}}(t)+\frac{t}{T}G^{(T)}\label{eq:1}
\end{eqnarray}
gives a non-Markovian coupling of two standard Brownian motions on
$\left[0,T\right]$ satisfying $B\left(T\right)=\widetilde{B}\left(T\right)$.
This coupling is similar in spirit to the one developed in \cite{BanerjeeKendall2016a}.
The usefulness of this coupling strategy arises when we want to couple
two copies of the process $\left(\left(B(t),F\left(\left[B\right]_{t}\right)\right):t\geqslant 0\right)$, where $B$ is a Brownian motion, $\left[B\right]_{t}$ denotes the
whole Brownian path up until time $t$ (thought of as an element of
$C\left[0,t\right]$), and $F$ is a (possibly random) functional
on $C\left[0, t \right]$. We first reflection couple the Brownian motions
until they meet. Then, by dividing the future time into intervals
$\left[T_{n}, T_{n+1}\right]$ (usually of growing length) and constructing
a suitable non-Markovian coupling of the Brownian bridges on each
such interval, we can obtain a coupling of the Brownian paths by the
above recipe in such a way that the corresponding path functionals
agree at one of the deterministic times $T_{n}$. As by construction,
the coupled Brownian motions agree at the times $T_{n}$, we achieve
a successful coupling of the joint process $\left(B,F\right)$. Further,
the rate of coupling attained by this non-Markovian strategy is usually
significantly better than Markovian strategies, and is often near
optimal (see \cite{BanerjeeKendall2016a}).

We will be interested in the particular choice of the random functional, namely,
\[
F\left(\left[w\right]_{t}\right)=\int_{0}^{t}w(s)dB_{1}(s),
\]
where $B_{1}$ is a standard Brownian motion and $w\in C\left[0,t\right]$.
Our coupling strategy for the Brownian bridges on $\left[0,T\right]$ will be based on the Karhunen-Lo\`{e}ve expansion which goes back to \cite{Karhunen1947a, Loeve1948} and for examples of such expansions see \cite[p.21]{WangLPhDThesis2008}. For the Brownian bridge we have
\begin{equation}
B^{\text{br}}(t)=\sqrt{T}\sum_{k=1}^{\infty}Z_{k}\frac{\sqrt{2}\sin\left(\frac{k\pi t}{T}\right)}{k\pi}=\sqrt{T}\sum_{k=1}^{\infty}Z_{k} g_{T, k}\left( t\right) \label{eq:2}
\end{equation}
for $t\in\left[0,T\right],$ where $Z_{k}$ are i.i.d. standard Gaussian random variables. Thus, in order to couple two Brownian bridges on $\left[0,T\right]$, we will couple the random variables $\left\{ Z_{k}\right\} _{k\geqslant 1}$. We now state and prove the following lemmas.

\begin{lemma}\label{lem:1}
There exists a non-Markovian coupling of the diffusions

\begin{align*}
& \left\{\left(B_{1}(t), B_{2}(t), a+\int_0^t B_{2}(s)dB_{1}(s)\right): t \geqslant 0 \right\},
\\
& \left\{ \left(\widetilde{B}_{1}(t), \widetilde{B}_{2}(t), \widetilde{a}+\int_0^t\widetilde{B}_{2}(s)d\widetilde{B}_{1}(s)\right): t \geqslant 0 \right\},
\\
& B_1(0)=\widetilde{B}_1(0)=b_{1}, B_2(0)=\widetilde{B}_2(0)=b_{2}, \text{ and } a >\widetilde{a},
\end{align*}for which the coupling time $\tau$ satisfies
\[
\mathbb{P}\left(\tau>t\right)\leqslant C\frac{\left(a-\widetilde{a}\right)}{t}
\]
for some constant $C>0$ that does not depend on the starting points
and $t\geqslant \left(a-\widetilde{a}\right)$.
\end{lemma}

\begin{proof}
We will write $I(t)=a+\int_{0}^{t}B_{2}(s)dB_{1}(s)$ and $\widetilde{I}(t)=\widetilde{a}+\int_{0}^{t}\widetilde{B}_{2}(s)d\widetilde{B}_{1}(s)$.
From Brownian scaling, it is clear that for any $r \in \mathbb{R}$, the following distributional equality holds
\begin{align}
\left(\frac{B_{1}(t)}{r},\right. & \left. \frac{B_{2}(t)}{r}, \frac{a+\int_{0}^{t}B_{2}(s)dB_{1}(s)}{r^2} \right) \label{e.scaling}
\\
& \quad \,{\buildrel d \over =}\, \left(B'_{1}(t/r^2), B'_{2}(t/r^2),
\frac{a}{r^2}+\int_{0}^{t/r^2}B'_{2}(s)dB'_{1}(s)\right), \notag
\end{align}
where $B'_{1}, B'_{2}$ are independent Brownian motions with $B'_{1}(0)=b_1/r$, $B'_{2}(0)=b_2/r$.
Thus we can assume $a-\widetilde{a}=1$. For the general case, we can obtain the corresponding coupling by applying the same coupling strategy to the scaled process  using \eqref{e.scaling} with $r=\sqrt{a-\widetilde{a}}$.

Let us divide the non-negative real line into  intervals $\left[2^{n}-1, 2^{n+1}-1\right], n \geqslant 0$. We will synchronously couple $B_{1}$ and $\widetilde{B}_{1}$ at all
times. Thus, we sample the same Brownian path for $B_{1}$ and $\widetilde{B}_{1}$.
\emph{Conditional on this Brownian path} $\left\{ B_{1}(t):t\geqslant 0\right\}$ we
describe the coupling strategy for $B_{2}$ and $\widetilde{B}_{2}$ inductively
on successive intervals. Suppose we have constructed the coupling
on $\left[0,2^{n}-1\right]$ in such a way that the coupled Brownian
motions $B_{2}$ and $\widetilde{B}_{2}$ satisfy $B_{2}(2^{n}-1)=\widetilde{B}_{2}(2^{n}-1)=b_{2}$
and $I(2^n-1)>\widetilde{I}(2^n-1)$.
Conditional on $\left\{ \left(B_{2}(t),\widetilde{B}_{2}(t)\right):t\leqslant 2^{n}-1\right\} $
and the whole Brownian path $B_{1}$, we will construct the coupling
of $B_{2}(t)-b_{2}$ and $\widetilde{B}_{2}(t)-b_{2}$ for $t\in\left[2^{n}-1,2^{n+1}-1\right]$.
To this end, we will couple two Brownian bridges $B^{\text{br}}$
and $\widetilde{B}^{\text{br}}$ on $\left[2^{n}-1,2^{n+1}-1\right]$,
then sample an independent Gaussian random variable $G^{(2^{n})}$
with mean zero, variance $2^n$ and finally use the recipe \eqref{eq:1}
to get the coupling of $B_{2}$ and $\widetilde{B}_{2}$ on $\left[2^{n}-1,2^{n+1}-1\right]$.

Let $\left(Z_{1}^{(n)}, Z_{2}^{(n)}, \dots\right)$ and $\left(\widetilde{Z}_{1}^{(n)}, \widetilde{Z}_{2}^{(n)},\dots\right)$
denote the Gaussian coefficients in the Karhunen-Lo\`{e}ve expansion \eqref{eq:2}
corresponding to $B^{\text{br}}$ and $\widetilde{B}^{\text{br}}$ respectively.
Sample i.i.d Gaussians $Z_{k}$ and set $Z_{k}^{(n)}=\widetilde{Z}_{k}^{(n)}=Z_{k}$
for $k\geqslant 2$. Now we construct the coupling of $Z_{1}^{(n)}$ and
$\widetilde{Z}_{1}^{(n)}$. Let $W^{(n)}$ be a standard Brownian motion
starting from zero, independent of $\left\{ \left(B_{2}(t),\widetilde{B}_{2}(t)\right):t\leqslant 2^{n}-1\right\}$, $\left\{ Z_{k}\right\} _{k\geqslant 2}$ and $B_{1}$. In what follows we will repeatedly use the following random functional
\begin{equation}\label{e.lambda}
\lambda_{n}\left( t \right)=\frac{2}{\pi}\int_{2^{n}-1}^{t}\sqrt{2}\sin\left(
\frac{\pi (s-2^n+1)}{2^{n}} \right)dB_{1}(s), 2^{n}-1 \leqslant t \leqslant 2^{n+1}-1.
\end{equation}
Define the random time $\sigma^{(n)}$ by

\[
\sigma^{(n)}=
\left\{
\begin{array}{ll}
 \inf\left\{ t\geqslant 0: W^{(n)}(t)=-\frac{\left(I(2^{n}-1)-\widetilde{I}(2^{n}-1)\right)}{\lambda_{n}\left( 2^{n+1}-1 \right)}
\right\}, & \text{ if } \lambda_{n}\left(2^{n+1}-1\right)\not=0,
 \\
  \infty, & \text{ otherwise. }
\end{array}
\right.
\]

%\begin{align*}
%& \sigma^{(n)}= \inf\left\{ t\geqslant %0:W^{(n)}(t)=-\frac{\left(I(2^{n}-1)-\widetilde{I}(2^{n}-1)\right)}{\lambda_{n}}
%\right.
%\\
%& \left.
%-\left(I(2^{n}-1)-\widetilde{I}(2^{n}-1)\right)\Bigg/\left(\frac{2}{\pi}\int_{2^{n}-1}^{2^{n+1}-1}\sqrt{2}\sin\left(\frac{\pi %(s-2^n+1)}{2^{n}}\right)dB_{1}(s)\right)\right\},
%\end{align*}
%when
%\[
%\int_{2^{n}-1}^{2^{n+1}-1}\sqrt{2}\sin\left(\pi (s-2^n+1)/2^{n}\right)dB_{1}(s)\neq0,
%\]
%and $\sigma^{(n)}=\infty$ otherwise.
As $\lambda_{n}\left( 2^{n+1}-1 \right)$ is a Gaussian random variable with mean zero and variance
\[
\frac{4}{\pi^{2}}\int_{2^{n}-1}^{2^{n+1}-1}2\sin^{2}\left( \frac{\pi (s-2^n+1)}{2^{n}} \right)ds = \frac{2^{n+2}}{\pi^{2}},
\]
the time $\sigma^{(n)}$ is finite for almost every realization of the Brownian
path $B_{1}$. Now, define $\widetilde{W}^{(n)}$ as follows
\[
\widetilde{W}^{(n)}(t)=\begin{cases}
-W^{(n)}(t) & \mbox{if }t\leqslant \sigma^{(n)}\\
W^{(n)}(t)-2W^{(n)}\left(\sigma^{(n)}\right) & \text{ if }t>\sigma^{(n)}.
\end{cases}
\]
Conditional on $\left\{ \left(B_{2}(t),\widetilde{B}_{2}(t)\right):t\leqslant 2^{n}-1\right\} ,$
$\left\{ Z_{k}\right\} _{k\geqslant 2}$ and $B_{1}$, $\sigma^{(n)}$ is
a stopping time for $W^{(n)}$. Thus $\widetilde{W}^{(n)}$ defined above is also
a Brownian motion independent of $\left\{ \left(B_{2}(t),\widetilde{B}_{2}(t)\right):t\leqslant 2^{n}-1\right\}$,
$\left\{ Z_{k}\right\}_{k\geqslant 2}$ and $B_{1}.$

Finally, we set $Z_{1}^{(n)}=2^{-n/2}W^{(n)}\left(2^{n}\right)$ and $\widetilde{Z}_{1}^{(n)}=2^{-n/2}\widetilde{W}^{(n)}\left(2^{n}\right)$.
Under this coupling we get
\begin{equation}\label{levypart}
I(t)-\widetilde{I}(t)=I\left(2^{n}-1\right)-\widetilde{I}\left(2^{n}-1\right)
+W^{(n)}\left(2^{n}\wedge\sigma^{(n)}\right) \lambda_{n}\left( t \right)
%\frac{2}{\pi}\int_{2^{n}-1}^{t}\sqrt{2}\sin\left(\frac{\pi (s-2^n+1)}{2^{n}}\right)dB_{1}(s)
,
\end{equation}
for $t\in\left[2^{n}-1,2^{n+1}-1\right]$. In particular, $I\left(2^{n+1}-1\right)-\widetilde{I}\left(2^{n+1}-1\right)\geqslant 0$
and equals to zero if and only if $\sigma^{(n)}\leqslant 2^{n}$.
%If $I\left(2^{n}-1\right)-\widetilde{I}\left(2^{n+1}-1\right)\geq0$
%and equals zero if and only if $\sigma^{(n)}\leqslant 2^{n}.$
If $I\left(2^{n}-1\right)-\widetilde{I}\left(2^{n}-1\right)=0,$
we synchronously couple $B_{2},\widetilde{B}_{2}$ after time $2^{n}-1$. By
induction, the coupling is defined for all time.

Now, we claim that the coupling constructed above gives the required bound on the coupling rate. %Define $\lambda_{n}=\frac{2}{\pi}\int_{2^{n}-1}^{2^{n+1}-1}\sqrt{2}\sin\left(\pi %(s-2^n+1)/2^{n}\right)dB_{1}(s)$.
Using L\'{e}vy's characterization of Brownian motion and the fact that
the $\left\{ W^{(n)}\right\} _{n\geqslant 1}$ are independent of the Brownian
path $B_{1},$ we obtain a Brownian motion $B^{\star}$ independent
of $B_{1}$ such that for all $t\geqslant 0$,
\[
\sum_{k=0}^{\infty}\lambda_{k}\left( 2^{k+1}-1 \right) W^{(k)}\left(\left(t-2^{k}+1\right)^{+}\wedge2^{k}\right)=B^{\star}\left(T(t)\right),
\]
where
\[
T(t)=\int_{0}^{t}\sum_{k=0}^{\infty}\lambda_{k}^{2}\left( 2^{k+1}-1 \right) \mathbbm{1}\left(2^{k}-1<s\leqslant 2^{k+1}-1\right)ds.
\]
Note that for any $n\geqslant 0,$ the coupling happens after time $2^{n+1}-1$
if and only if $\sigma^{(k)}>2^{k}$ for all $k\leqslant  n,$ that is, $B^{\star}(t)>\left(\widetilde{a}-a\right)=-1$
for all $t\leqslant  T\left(2^{n+1}-1\right)$. Therefore, if for $y\in\mathbb{R},$
$\tau_{y}^{\star}$ denoted the hitting time of level $y$ for the
Brownian motion $B^{\star},$ then we have
\[
\mathbb{P}\left(\tau>2^{n+1}-1\right)=\mathbb{P}\left(\tau_{-1}^{\star}>T\left(2^{n+1}-1\right)\right).
\]
By a standard hitting time estimate for Brownian motion, we see that there is
a constant $C>0$ that does not depend on $b_1, b_2, a, \widetilde{a}$ such that
\begin{equation}\label{eq:3}
\mathbb{P}\left(\tau>2^{n+1}-1\right)\leqslant  C\mathbb{E}\left[\frac{1}{\sqrt{T\left(2^{n+1}-1\right)}}\right].
\end{equation}
Thus, we need to obtain an estimate for the right hand side in \eqref{eq:3}. Note that $2^{-2n}T\left(2^{n+1}-1\right)$ has the same distribution as
\[
\Psi_{n}:=\frac{4}{\pi^2}\sum_{k=0}^{n}2^{-2k}U_{k}^{2},
\] where the $U_{k}$ are i.i.d. standard Gaussian random variables.
%and set
%\[
%\Psi:=\frac{4}{\pi^2}\sum_{k=0}^{\infty}2^{-2k}U_{k}^{2}.
%\]
%Then $\mathbb{E}\left[\Psi\right]<\infty$ and thus, $\Psi$ is almost surely finite.

%Note that $\Psi_{n}^{-1/2}\to\Psi^{-1/2}$ almost surely.
For $n\geqslant 1,$
$\Psi_{n}^{-1/2}\leqslant \Psi_{1}^{-1/2}\leqslant  \pi \left(U_{0}^{2}+U_{1}^{2}\right)^{-1/2}$. As
$\sqrt{U_{0}^{2}+U_{1}^{2}}$ has density $re^{-r^{2}/2}dr$ with respect to the Lebesgue measure for $r \geqslant 0$, we conclude that
$\mathbb{E}\left[\pi\left(U_{0}^{2}+U_{1}^{2}\right)^{-1/2}\right]<\infty.$
Thus, for $n \geqslant  1$
\[
\mathbb{E}\left[\frac{1}{\sqrt{2^{-2n}T\left(2^{n+1}-1\right)}}\right] = \mathbb{E}\left[\Psi_{n}^{-1/2}\right] \leqslant \mathbb{E}\left[\Psi_{1}^{-1/2}\right] \leqslant  \mathbb{E}\left[\pi \left(U_{0}^{2}+U_{1}^{2}\right)^{-1/2}\right] <\infty.
\]
%by the dominated convergence theorem, $\mathbb{E}\left[\Psi^{-1/2}\right]$
%is finite and $\mathbb{E}\left[\Psi_{n}^{-1/2}\right]\to\mathbb{E}\left[\Psi^{-1/2}\right]$
%as $n\to\infty$.
This, along with \eqref{eq:3}, implies that there is a positive constant $C$ not depending on $b_1, b_2, a,\widetilde{a}$ such that for $n \geqslant 1$,
\[
\mathbb{P}\left(\tau>2^{n+1}-1\right)\leqslant \frac{C}{2^{n}}.
\]
It is easy to check that the above inequality implies the lemma.
\end{proof}
\begin{rem}\label{rem:Mark}
Under the hypothesis of Lemma \ref{lem:1}, it is \textit{not possible to obtain the given rate of decay} of the probability of failing to couple by time $t$ (coupling rate) with \textit{any Markovian coupling}. The proof of this proceeds similar to that of \cite[Lemma 3.1]{BanerjeeKendall2016a}. We sketch it here. Under any Markovian coupling $\mu$, a simple Fubini argument shows that there exists a deterministic time $t_0>0$ such that $\mu\left(B(t_0) \neq \widetilde{B}(t_0)\right) >0$. Let $\tau^B$ represent the first time when the Brownian motions $B$ and $\widetilde{B}$ meet after time $t_0$ (which should happen at or before the coupling time of $\mathbf{X}$ and $\widetilde{\mathbf{X}}$). Let $\mathcal{F}_{t_0}$ denote the filtration generated by $B$ and $\widetilde{B}$ up to time $t_0$ and let $\mathbb{E}_{\mu}$ denote expectation under the coupling law $\mu$. Then, from the fact that the maximal coupling rate of Brownian motion (equivalently the total variation distance between $B(t)$ and $\widetilde{B}(t)$) decays like $t^{-1/2}$, we deduce that for sufficiently large $t$
\begin{align*}
\mu(\tau>t) &= \mathbb{E}_{\mu} \mathbb{E}_{\mu}\left[\tau>t \mid \mathcal{F}_{t_0}\right] \geqslant \mathbb{E}_{\mu} \mathbb{E}_{\mu}\left[\tau^B>t \mid \mathcal{F}_{t_0}\right]\\
&\geqslant C_{\mu}(t-t_0)^{-1/2} \geqslant C_{\mu}t^{-1/2},
\end{align*}
where $C_{\mu}$ denotes a positive constant that depends on the coupling $\mu$. Thus, any Markovian coupling has coupling rate at least $t^{-1/2}$, but the non-Markovian coupling described in Lemma \ref{lem:1} gives a rate of $t^{-1}$.
\end{rem}
The next lemma gives an estimate of the tail of the law of the stochastic integral $\int_{0}^{t}B_{2}(s)dB_{1}(s)$ run until
the first time $B_{2}$ hits zero.

\begin{lemma} \label{lem:2}
Let $B_{1}, B_{2}$ be independent Brownian motions with $B_{2}(0)=b>0$.
For $z\in\mathbb{R}$, let $\tau_{z}$ denote the hitting time of
level $z$ by $B_{2}$. Then
\[
\mathbb{P}\left(\int_{0}^{\tau_{0}}B_{2}(s)dB_{1}(s)>y\right)\leqslant \frac{2b}{\sqrt{y}} \hskip0.1in \text{ for } y\geqslant  b^{2}.
\]
\end{lemma}

\begin{proof}
For any level $z\geqslant  b$, we can write
\begin{align*}
& \mathbb{P}\left(\int_{0}^{\tau_{0}}B_{2}(s)dB_{1}(s)>y\right)  =
\\
& \mathbb{P}\left(\int_{0}^{\tau_{0}}B_{2}(s)dB_{1}(s)>y, \tau_{z}<\tau_{0}\right)+
 \mathbb{P}\left(\int_{0}^{\tau_{0}}B_{2}(s)dB_{1}(s)>y, \tau_{z}\geqslant \tau_{0}\right) \leqslant
\\
&\mathbb{P}\left(\tau_{z}<\tau_{0}\right)+\frac{\mathbb{E}\left[\int_{0}^{\tau_{0}\wedge\tau_{z}}B_{2}^{2}(s)ds\right]}{y^{2}}
  \leqslant
  \\& \mathbb{P}\left(\tau_{z}<\tau_{0}\right)+\frac{z^{2}}{y^{2}}\mathbb{E}\left[\tau_{0}\wedge\tau_{z}\right],
\end{align*}
where the second step follows from Chebyshev's inequality. From standard estimates for Brownian motion, $\mathbb{P}\left(\tau_{z}<\tau_{0}\right)=b/z$
and $\mathbb{E}\left[\tau_{0}\wedge\tau_{z}\right]=b(z-b) \leqslant bz$. Using these
in the above, we get
\[
\mathbb{P}\left(\int_{0}^{\tau_{0}}B_{2}(s)dB_{1}(s)>y\right)\leqslant \frac{b}{z}+\frac{bz^{3}}{y^{2}}.
\]
As this bound holds for arbitrary $z\geqslant  b$, the result follows
by choosing $z=\sqrt{y}$.
\end{proof}
%The next lemma gives an estimate on $\mathbb{E}\left|\frac{A(T_1)}{t} \wedge 1\right|$ for large enough $t$.
%\begin{lemma}\label{lem:Aest}
%For $t\geqslant \max\left\{ \left|\mathbf{b}-\widetilde{\mathbf{b}}\right|^{2},2\left|a-\widetilde{a}+b_{1}\widetilde{b}_{2}-b_{2}\widetilde{b}_{1}\right|\right\} $, there exists a positive constant not depending on the starting points such that
%$$
%\mathbb{E}\left|\frac{A(T_1)}{t} \wedge 1\right| \leqslant C\left(\frac{\left|\mathbf{b}-\widetilde{\mathbf{b}}\right|}{\sqrt{t}}+\frac{\left|a-\widetilde{a}+b_{1}\widetilde{b}_{2}-b_{2}\widetilde{b}_{1}\right|}{t}\right).
%$$
%\end{lemma}
Consider two coupled Brownian motions $\left(\mathbf{X},\widetilde{\mathbf{X}}\right)$ on the Heisenberg group starting from $(b_{1}, b_{2},a)$ and $\left(b_{1},\widetilde{b}_{2}, \widetilde{a}\right)$ respectively. A key object in our coupling construction for Brownian motions on the Heisenberg group $\mathbb{H}^{3}$ will be the \emph{invariant difference of stochastic areas given} by
\begin{align}\label{align:inv}
A(t) & =
\left(a-\widetilde{a}\right)
+\left(\int_{0}^{t}B_{1}(s)dB_{2}(s)-\int_{0}^{t}B_{2}(s)dB_{1}(s)\right)
\\
&
-\left(\int_{0}^{t}\widetilde{B}_{1}(s)d\widetilde{B}_{2}(s)-\int_{0}^{t}\widetilde{B}_{2}(s)d\widetilde{B}_{1}(s)\right)\notag
 + B_{1}(t)\widetilde{B}_{2}(t)-B_{2}(t)\widetilde{B}_{1}(t).
\end{align}

Note that the L\'{e}vy stochastic area is invariant under rotations of
coordinates. If the Brownian motions $B_{1}$ and $\widetilde{B}_{1}$ are synchronously coupled at all times, then as the covariation between $B_1$ and $B_2$ (and between $B_1$ and $\widetilde{B}_2$) is zero,
\begin{equation}\label{invito}
A(t) - A(0)=-2\int_0^tB_{2}(s)dB_{1}(s) + 2\int_0^t\widetilde{B}_{2}(s)dB_1(s),
\end{equation}
where
\begin{equation}\label{Azero}
A(0) = a-\widetilde{a}+b_{1}\widetilde{b}_{2}-b_{2}\widetilde{b}_{1},
\end{equation}
for $t\geqslant 0$. The next lemma establishes a control on the invariant difference evaluated at the time when the Brownian motions $B_2$ and $\widetilde{B}_2$ first meet, provided they are reflection coupled up to that time.
\begin{lemma}\label{invcontrol}
Let $B_1$ be a real-valued Brownian motion starting from $b_1$, and let $B_2, \widetilde{B}_2$ be reflection coupled one-dimensional Brownian motions starting from $b_2$ and $\widetilde{b}_2$ respectively. Consider the invariant difference of stochastic areas given by \eqref{align:inv} with $B_1= \widetilde{B}_1$. Define
$
T_{1}=\inf\left\{ t\geqslant 0:B_{2}(t)=\widetilde{B}_{2}(t)\right\}.
$
Then there exists a positive constant $C$ that does not depend on $b_1, b_2, \widetilde{b}_2, a, \widetilde{a}$ such that for any $t\geqslant \max\left\{ \left|b_2-\widetilde{b}_2\right|^{2},2\left|a-\widetilde{a}+b_{1}\widetilde{b}_{2}-b_{2}\widetilde{b}_{1}\right|\right\} $,
\begin{equation*}
\mathbb{E}\left[\frac{\left|A\left(T_{1}\right)\right|}{t}\wedge1\right] \leqslant  C\left(\frac{\left|b_2-\widetilde{b}_2\right|}{\sqrt{t}}+\frac{\left|a-\widetilde{a}+b_{1}\widetilde{b}_{2}-b_{2}\widetilde{b}_{1}\right|}{t}\right).
\end{equation*}
\end{lemma}

\begin{proof}
In the proof, $C, C'$ will denote generic positive constants that do not depend on $b_1, b_2, \widetilde{b}_2, a, \widetilde{a}$, whose values might change from line to line. For any $t>0$,
\begin{align}\label{ar}
\mathbb{E}\left[\frac{\left|A\left(T_{1}\right)\right|}{t}\wedge1\right]
 & \leqslant  \sum_{k=0}^{\infty}\mathbb{E}\left[\frac{\left|A\left(T_{1}\right)\right|}{t}\wedge1;2^{-k-1}t < \left|A\left(T_{1}\right)\right|\leqslant 2^{-k}t\right]+\mathbb{P}\left(\left|A\left(T_{1}\right)\right|>t\right)\nonumber\\
 & \leqslant  \sum_{k=0}^{\infty}2^{-k}\mathbb{P}\left(2^{-k-1}t < \left|A\left(T_{1}\right)\right|\leqslant 2^{-k}t\right) + \mathbb{P}\left(\left|A\left(T_{1}\right)\right|>t\right)\nonumber\\
 & \leqslant  \sum_{k=0}^{\infty}2^{-k}\mathbb{P}\left(\left|A\left(T_{1}\right)\right|\geqslant 2^{-k-1}t\right)+\mathbb{P}\left(\left|A\left(T_{1}\right)\right|>t\right).
\end{align}
As $B_2$ and $\widetilde{B}_2$ are reflection coupled, we can rewrite \eqref{invito} as
\[
A(t) - A(0)=-2\int_0^t\left(B_{2}(s)-\widetilde{B}_{2}(s)\right)dB_1(s)
\]
where $\frac{1}{2}\left(B_{2}-\widetilde{B}_{2}\right)$ is a Brownian motion starting from $\frac{1}{2}\left(b_2 - \widetilde{b}_2\right)$ and independent of $B_1$. By Lemma $\ref{lem:2}$, for $t\geqslant\max\left\{ \left|b_{2}-\widetilde{b}_{2}\right|^{2},2\left|A(0)\right|\right\} ,$
\begin{align}\label{ar1}
\mathbb{P}\left(\left|A\left(T_{1}\right)\right|>t\right) & \leqslant \mathbb{P}\left(\left|A\left(T_{1}\right)-A\left(0\right)\right|>t-\left|A\left(0\right)\right|\right)\nonumber\\
 & \leqslant \mathbb{P}\left(\left|A\left(T_{1}\right)-A\left(0\right)\right|>\frac{t}{2}\right)\leqslant  C\frac{\left|b_{2}-\widetilde{b}_{2}\right|}{\sqrt{t}}.
\end{align}
Further, for $t\geqslant \max\left\{ \left|b_{2}-\widetilde{b}_{2}\right|^{2}, 2\left|A(0)\right|\right\} ,$
\begin{multline}\label{area1}
\sum_{k=0}^{\infty}2^{-k}\mathbb{P}\left(\left|A\left(T_{1}\right)\right|\geqslant 2^{-k-1}t\right)\\
=  \sum_{k:2^{-k-1}t\leqslant \max\left\{ \left|b_{2}-\widetilde{b}_{2}\right|^{2},2\left|A(0)\right|\right\} }2^{-k}\mathbb{P}\left(\left|A\left(T_{1}\right)\right|\geqslant 2^{-k-1}t\right)\\
 +\sum_{k:2^{-k-1}t>\max\left\{ \left|b_{2}-\widetilde{b}_{2}\right|^{2},2\left|A(0)\right|\right\} }2^{-k}\mathbb{P}\left(\left|A\left(T_{1}\right)\right|\geqslant 2^{-k-1}t\right).
% & \leqslant  & C\frac{2\max\left\{ \left|b_{2}-\widetilde{b}_{2}\right|^{2},2\left|A(0)\right|\right\} }{t}\\
% &  & +\frac{C}{\sqrt{t}}\sum_{k:2^{-k-1}t>\max\left\{ \left|b_{2}-\widetilde{b}_{2}\right|^{2},2\left|A(0)\right|\right\} }2^{-k/2}\left|b_{2}-\widetilde{b}_{2}\right|\\
% & \leqslant  & C\left(\frac{\left|A\left(0\right)\right|}{t}+\frac{\left|b_{2}-\widetilde{b}_{2}\right|}{\sqrt{t}}\right).
\end{multline}
To estimate the first term on the right hand side of \eqref{area1}, let $k_0$ be the smallest integer $k$ such that $2^{-k-1}t \leqslant \max\left\{ \left|b_{2}-\widetilde{b}_{2}\right|^{2},2\left|A(0)\right|\right\}$. Then,
\begin{multline}\label{area2}
\sum_{k:2^{-k-1}t \leqslant \max\left\{ \left|b_{2}-\widetilde{b}_{2}\right|^{2},2\left|A(0)\right|\right\} }2^{-k}\mathbb{P}\left(\left|A\left(T_{1}\right)\right|\geqslant 2^{-k-1}t\right)\\
\leqslant \sum_{k=k_0}^{\infty}2^{-k} = 2^{-k_0+1}
= \frac{4}{t} 2^{-k_0-1}t \leqslant  \frac{4}{t} \max\left\{ \left|b_{2}-\widetilde{b}_{2}\right|^{2},2\left|A(0)\right|\right\}\\
\leqslant  8 \left(\frac{\left|b_{2}-\widetilde{b}_{2}\right|^{2}}{t} + \frac{|A(0)|}{t}\right) \leqslant 8 \left(\frac{\left|b_{2}-\widetilde{b}_{2}\right|}{\sqrt{t}} + \frac{\left|a-\widetilde{a}+b_{1}\widetilde{b}_{2}-b_{2}\widetilde{b}_{1}\right|}{t}\right),
\end{multline}
where we used the facts that
$\frac{\left|b_{2}-\widetilde{b}_{2}\right|^2}{t}\leqslant \frac{\left|b_{2}-\widetilde{b}_{2}\right|}{\sqrt{t}}$
for $t\geqslant \left|b_{2}-\widetilde{b}_{2}\right|^{2}$ and $A(0) = a-\widetilde{a}+b_{1}\widetilde{b}_{2}-b_{2}\widetilde{b}_{1}$ to get the last inequality.

To estimate the second term on the right hand side of \eqref{area1}, we use Lemma \ref{lem:2} to get
\begin{multline}\label{area3}
\sum_{k:2^{-k-1}t>\max\left\{ \left|b_{2}-\widetilde{b}_{2}\right|^{2},2\left|A(0)\right|\right\} }2^{-k}\mathbb{P}\left(\left|A\left(T_{1}\right)\right|\geqslant 2^{-k-1}t\right)\\
\leqslant  \frac{C}{\sqrt{t}}\sum_{k:2^{-k-1}t>\max\left\{ \left|b_{2}-\widetilde{b}_{2}\right|^{2},2\left|A(0)\right|\right\} }2^{-k/2}\left|b_{2}-\widetilde{b}_{2}\right|\\
\leqslant  \frac{C\left|b_{2}-\widetilde{b}_{2}\right|}{\sqrt{t}} \sum_{k=0}^{\infty}2^{-k/2} \leqslant  C'\frac{\left|b_{2}-\widetilde{b}_{2}\right|}{\sqrt{t}}.
\end{multline}
Using \eqref{area2} and \eqref{area3} in \eqref{area1},
\begin{equation}\label{ar2}
\sum_{k=0}^{\infty}2^{-k}\mathbb{P}\left(\left|A\left(T_{1}\right)\right|\geqslant 2^{-k-1}t\right) \leqslant  C\left(\frac{\left|b_2-\widetilde{b}_2\right|}{\sqrt{t}}+\frac{\left|a-\widetilde{a}+b_{1}\widetilde{b}_{2}-b_{2}\widetilde{b}_{1}\right|}{t}\right).
\end{equation}
Using \eqref{ar1} and \eqref{ar2} in \eqref{ar}, we complete the proof of the lemma.
\end{proof}

Now, we state and prove our main theorem on coupling of Brownian motions on the Heisenberg group $\mathbb{H}^{3}$.

%\masha{I reformulated this theorem slightly. Please check.}
\begin{theorem}\label{thm:3} There exists a non-Markovian coupling $\left(\mathbf{X},\widetilde{\mathbf{X}}\right)$ of two Brownian motions on the Heisenberg group starting from $(b_{1}, b_{2},a)$
and $\left(\widetilde{b}_{1},\widetilde{b}_{2}, \widetilde{a}\right)$ respectively,
 and a constant $C>0$ which does not depend on the starting points such that the coupling time $\tau$ satisfies
\[
\mathbb{P}\left(\tau>t\right)\leqslant C\left(\frac{\left|\mathbf{b}-\widetilde{\mathbf{b}}\right|}{\sqrt{t}}+\frac{\left|a-\widetilde{a}+b_{1}\widetilde{b}_{2}-b_{2}\widetilde{b}_{1}\right|}{t}\right)
\]
for $t\geqslant \max\left\{ \left|\mathbf{b}-\widetilde{\mathbf{b}}\right|^{2},2\left|a-\widetilde{a}+b_{1}\widetilde{b}_{2}-b_{2}\widetilde{b}_{1}\right|\right\} $. Here $\mathbf{b}=\left(b_{1},b_{2}\right)$ and $\widetilde{\mathbf{b}}=\left(\widetilde{b}_{1}, \widetilde{b}_{2}\right)$.
\end{theorem}

\begin{proof}
We will explicitly construct the non-Markovian coupling. In the proof, $C$ will denote a generic positive constant that does not depend
on the starting points.

Since the L\'{e}vy stochastic area is invariant under rotations of
coordinates, it suffices to consider the case when $b_{1}=\widetilde{b}_{1}$.
Recall the \emph{invariant difference of stochastic areas} $A$ defined by \eqref{align:inv}.
%\begin{align}\label{align:inv}
%A(t) & =
%\left(a-\widetilde{a}\right)
%+\left(\int_{0}^{t}B_{1}(s)dB_{2}(s)-\int_{0}^{t}B_{2}(s)dB_{1}(s)\right)
%\\
%&
%-\left(\int_{0}^{t}\widetilde{B}_{1}(s)d\widetilde{B}_{2}(s)-\int_{0}^{t}\widetilde{B}_{2}(s)d\widetilde{B}_{1}(s)\right)\notag
% + B_{1}(t)\widetilde{B}_{2}(t)-B_{2}(t)\widetilde{B}_{1}(t).
%\end{align}
We will synchronously couple the Brownian motions $B_{1}$ and $\widetilde{B}_{1}$ at all times. Recall that under this setup, the invariant difference takes the form \eqref{invito}.
%Now we synchronously couple the Brownian motions $B_{1}$ and $\widetilde{B}_{1}$ at all times. Thus, we have
%\[
%dA(t)=-2\left(B_{2}(t)-\widetilde{B}_{2}(t)\right)dB_{1}(t)
%\]
%for $t\geqslant 0$.
The coupling comprises the following two steps.

\emph{Step} 1. We use a reflection coupling for $B_{2}$ and $\widetilde{B}_{2}$ until
the first time they meet. Let $T_{1}=\inf\left\{ t\geqslant 0:B_{2}(t)=\widetilde{B}_{2}(t)\right\} .$

\emph{Step} 2.  After time $T_{1}$ we apply the coupling strategy described in Lemma $\ref{lem:1}$ to the diffusions

\begin{align*}
& \left\lbrace\left(B_1(t), B_2(t), A(T_1)+\int_{T_1}^tB_2(s)dB_1(s)\right): t \geqslant T_1\right\rbrace,
\\
& \left\lbrace\left(\widetilde{B}_1(t), \widetilde{B}_2(t), \int_{T_1}^t\widetilde{B}_2(s)d\widetilde{B}_1(s)\right): t \geqslant T_1\right\rbrace.
\end{align*}

By standard estimates for the Brownian hitting time we have
\begin{equation}\label{fund1}
\mathbb{P}\left(T_{1}>t\right)\leqslant \frac{C\left|b_{2}-\widetilde{b}_{2}\right|}{\sqrt{t}}
\end{equation}
for $t\geqslant \left|b_{2}-\widetilde{b}_{2}\right|^{2}$. By Lemma \ref{lem:1} and Lemma \ref{invcontrol}, for
$t\geqslant \max\left\{ \left|b_{2}-\widetilde{b}_{2}\right|^{2},2\left|A(0)\right|\right\} $,
\begin{multline}\label{fund2}
\mathbb{P}\left(\tau-T_{1}>t\right) \leqslant  C\mathbb{E}\left[\frac{\left|A\left(T_{1}\right)\right|}{t}\wedge1\right]\\
\leqslant C\left(\frac{\left|b_2-\widetilde{b}_2\right|}{\sqrt{t}}+\frac{\left|a-\widetilde{a}+b_{1}\widetilde{b}_{2}-b_{2}\widetilde{b}_{1}\right|}{t}\right).
\end{multline}
Equations \eqref{fund1} and \eqref{fund2} together yield the required tail bound on the coupling time probability stated in the theorem.
\end{proof}
An interesting observation to note from Theorem \ref{thm:3} is that, if the Brownian motions start from the same point, then the coupling rate is significantly faster.

The above coupling can be used to get sharp estimates on the total variation distance between the laws of two Brownian motions on the Heisenberg group starting from distinct points.
\begin{thm}\label{thm:TVD}
If $d_{TV}$ denotes the total variation distance between probability
measures, and $\mathcal{L}\left(\mathbf{X}_{t}\right),\mathcal{L}\left(\widetilde{\mathbf{X}}_{t}\right)$
denote the laws of Brownian motions on the Heisenberg group starting from $\left(b_{1},b_{2},a\right)$
and $\left(\widetilde{b}_{1},\widetilde{b}_{2},\widetilde{a}\right)$ respectively,
then there exists positive constants $C_1$, $C_2$ not depending on the starting points such that
\begin{align*}
d_{TV}\left(\mathcal{L}\left(\mathbf{X}_{t}\right),\mathcal{L}\left(\widetilde{\mathbf{X}}_{t}\right)\right)&\leqslant  C_1\left(\frac{\left|\mathbf{b}-\widetilde{\mathbf{b}}\right|}{\sqrt{t}}+\frac{\left|a-\widetilde{a}+b_{1}\widetilde{b}_{2}-b_{2}\widetilde{b}_{1}\right|}{t}\right)\\
d_{TV}\left(\mathcal{L}\left(\mathbf{X}_{t}\right),\mathcal{L}\left(\widetilde{\mathbf{X}}_{t}\right)\right) &\geqslant C_2\left(\frac{\left|\mathbf{b}-\widetilde{\mathbf{b}}\right|}{\sqrt{t}}\mathbbm{1}(\mathbf{b}\neq\widetilde{\mathbf{b}}) + \frac{\left|a-\widetilde{a}\right|}{t}\mathbbm{1}(\mathbf{b}=\widetilde{\mathbf{b}})\right)
\end{align*}
for $t\geqslant \max\left\{ \left|\mathbf{b}-\widetilde{\mathbf{b}}\right|^{2},2\left|a-\widetilde{a}+b_{1}\widetilde{b}_{2}-b_{2}\widetilde{b}_{1}\right|\right\} $.
\end{thm}

\begin{proof}
The upper bound on the total variation distance follows from Theorem $\ref{thm:3}$ and the Aldous'
inequality \eqref{equation:Aldous}.

To prove the lower bound, we first address the case $\mathbf{b} \neq \widetilde{\mathbf{b}}$. It is straightforward to see from the definition of the total variation distance that
\[
d_{TV}\left(\mathcal{L}\left(\mathbf{X}_{t}\right),\mathcal{L}\left(\widetilde{\mathbf{X}}_{t}\right)\right) \geqslant d_{TV}\left(\mathcal{L}\left(\mathbf{B}_{t}\right),\mathcal{L}\left(\widetilde{\mathbf{B}}_{t}\right)\right).
\]
Thus, when $\mathbf{b} \neq \widetilde{\mathbf{b}}$, the lower bound in the theorem follows from the standard estimate on the total variation distance between the laws of Brownian motions using the reflection principle
\[
 d_{TV}\left(\mathcal{L}\left(\mathbf{B}_{t}\right),\mathcal{L}\left(\widetilde{\mathbf{B}}_{t}\right)\right) = \mathbb{P}\left(|N(0,1)| \leqslant \frac{\left|\mathbf{b}-\widetilde{\mathbf{b}}\right|}{2\sqrt{t}}\right)\geqslant \frac{1}{\sqrt{2\pi e}}\frac{\left|\mathbf{b}-\widetilde{\mathbf{b}}\right|}{\sqrt{t}}.
\]
where $N(0,1)$ denotes a standard Gaussian variable.

Now, we deal with the case $\mathbf{b} = \widetilde{\mathbf{b}}$. As the generator of Brownian motion on the Heisenberg group  is hypoelliptic, the law  of Brownian motion starting from $(u,v,w)$ has a density with respect to the Lebesgue measure on $\mathbb{R}^3$ which coincides with the Haar measure on $\mathbb{H}^{3}$. We denote by $p^{\left( u, v, w \right)}_t(\cdot, \cdot, \cdot)$ this density (the heat kernel)  at time $t$. The heat kernel  $p^{(u,v,w)}_{t}\left( x, y, z \right)$ is a symmetric function of $\left( (u,v,w), \left( x, y, z \right) \right) \in \mathbb{H}^{3}\times \mathbb{H}^{3}$ and is invariant under left multiplication, that is, $p_{t}^{\left( u, v, w \right)}\left( x, y, z \right) = p_t^{e}(\left( u, v, w \right)^{-1}\left( x, y, z \right))=p_t^{e}(\left( x, y, z \right)\left( u, v, w \right)^{-1})$. Using the fact that $\left( u, v, w \right)^{-1}=\left( -u, -v, -w \right)$ we see that

\begin{equation}\label{repTV}
p^{(u,v,w)}_t(x,y,z)= p^{e}_t(x-u,y-v, z - w -uy+ vx), \text{ where } e=(0,0,0).
\end{equation}

%Then, by a simple change of variable formula, we deduce
%\begin{equation}\label{repTV}
%p^{(u,v,w)}_t(x,y,z)= p^{(0,0,0)}_t(x-u,y-v, z - w -u(y-v)+ v(x-u)).
%\end{equation}
%To see this, we use an explicit representation \eqref{Heiseneq} of a Brownian motion on the %Heisenberg group $\mathbf{X} = (X_1, X_2, X_3)$ starting from $(u, v, w)$, which can be rewritten %in the following form
%\begin{align*}
%& X_{1}(t) - u  =  B^{\prime}_{1}(t),\\
%& X_{2}(t) - v  =  B^{\prime}_{2}(t),\\
%& X_{3}(t) - w - u(X_{2}(t) - v) + v(X_{1}(t) - u)  =  %\int_{0}^{t}B^{\prime}_{1}(s)dB^{\prime}_{2}(s)-\int_{0}^{t}B^{\prime}_{2}(s)dB^{\prime}_{1}(s),
%\end{align*}
%where $B^{\prime}_{1}, B^{\prime}_{2}$ are independent Brownian motions, each starting from zero. %The above representation gives the appropriate change of co-ordinates leading to \eqref{repTV}.
%Denoting the Lebesgue measure on $\mathbb{R}^3$ by $\lambda$,
Then
\begin{align*}
& d_{TV}\left(\mathcal{L}\left(\mathbf{X}_{t}\right),\mathcal{L}\left(\widetilde{\mathbf{X}}_{t}\right)\right) = \int_{\mathbb{R}^3}\left|p^{(b_1, b_2, a)}_t(x, y, z)- p^{(b_1, b_2, \widetilde{a})}_t(x,y,z)\right| dx dy dz
\\
& = \int_{\mathbb{R}^3}\left|p^{e}_t(x-b_1,y-b_2, z - a -b_1y+ b_2x)
\right.
\\
& \left.
- p^{e}_t(x-b_1,y-b_2, z- \widetilde{a} -b_1y+ b_2x)\right|  dx dy dz
\\
& = \int_{\mathbb{R}^3}\left|p^{e}_t(x,y, z - a)- p^{e}_t(x,y, z- \widetilde{a})\right| dx dy dz
\\
& \geqslant \int_{\mathbb{R}} \left|f_t(z-a)-f_t(z-\widetilde{a})\right|dz,
\end{align*}
where $f_t$ denotes the density with respect to the Lebesgue measure of the L\'{e}vy stochastic area at time $t$ when the driving Brownian motion starts at the origin. The third equality above follows by a simple change of variable formula and the last step follows from two applications of the inequality $\left|\int_{\mathbb{R}}f(x)dx\right| \leqslant \int_{\mathbb{R}}|f(x)|dx$ for real-valued measurable $f$.

From Brownian scaling, it is easy to see that
\[
f_t(z)=\frac{1}{t}f_1\left(\frac{z}{t}\right), \ \ z \in \mathbb{R}.
\]
Substituting this in the above and using the change of variable formula again, we get
\begin{align*}
d_{TV}\left(\mathcal{L}\left(\mathbf{X}_{t}\right),\mathcal{L}\left(\widetilde{\mathbf{X}}_{t}\right)\right) &\geqslant  \int_{\mathbb{R}} \left|f_1\left(z-\frac{a}{t}\right)-f_1\left(z-\frac{\widetilde{a}}{t}\right)\right|dz\\
&= \int_{\mathbb{R}} \left|f_1\left(z-\frac{a-\widetilde{a}}{t}\right)-f_1\left(z\right)\right|dz\\
&\geqslant \int_{|z| \geqslant 1} \left|f_1\left(z-\frac{a-\widetilde{a}}{t}\right)-f_1\left(z\right)\right|dz.
\end{align*}
The explicit form of $f_1$ is well-known (see, for example, \cite{Yor1991a} or \cite[p. 32]{NeuenschwanderLNM1996})
\[
f_1(z)=\frac{1}{\cosh \pi z}, \ \ z \in \mathbb{R}.
\]
Without loss of generality, we assume $a > \widetilde{a}$. By the mean value theorem and the assumption made in the theorem that $\frac{a-\widetilde{a}}{t} \leqslant \frac{1}{2}$,
\begin{align*}
\left|f_1\left(z-\frac{a-\widetilde{a}}{t}\right)-f_1\left(z\right)\right| &\geqslant \frac{a-\widetilde{a}}{t}\inf_{\zeta \in \left[z-\frac{a-\widetilde{a}}{t},z\right]}|f_1'(\zeta)|\\
&\geqslant \frac{a-\widetilde{a}}{t}\inf_{\zeta \in \left[z-\frac{1}{2},z\right]}|f_1'(\zeta)|.
\end{align*}
We can explicitly compute
\[
|f_1'(\zeta)| = \frac{2\pi |e^{\pi\zeta}- e^{-\pi\zeta}|}{(e^{\pi\zeta} + e^{-\pi\zeta})^2}.
\]
This is an even function which is strictly decreasing for $\zeta \geqslant 1/2$. Thus, for $|z| \geqslant 1$,
\[
\inf_{\zeta \in \left[z-\frac{1}{2}, z\right]}|f_1^{\prime}(\zeta)| \geqslant |f_1^{\prime}(3z/2)|.
\]
Thus,
\begin{align*}
d_{TV}\left(\mathcal{L}\left(\mathbf{X}_{t}\right),\mathcal{L}\left(\widetilde{\mathbf{X}}_{t}\right)\right)  & \geqslant \int_{|z| \geqslant 1} \left|f_1\left(z-\frac{a-\widetilde{a}}{t}\right)-f_1\left(z\right)\right|dz\\
& \geqslant \frac{|a-\widetilde{a}|}{t}\int_{|z| \geqslant 1}|f_1'(3z/2)|dz = C_2 \frac{|a-\widetilde{a}|}{t},
\end{align*}
which completes the proof of the theorem.
\end{proof}
Several remarks are in order.
\begin{rem}\label{rem:efficiency}
Theorem \ref{thm:TVD} shows that the non-Markovian coupling strategy we constructed is, in fact, an \textit{efficient coupling strategy} in the sense that the coupling rate decays according to the same power of $t$ as the total variation distance between the laws of the Brownian motions $\mathbf{X}$ and $\widetilde{\mathbf{X}}$. We refer to \cite[Definition 1]{BanerjeeKendall2016a} for the precise notion of efficiency.
\end{rem}

\begin{rem}\label{rem:constants}
Although we have stated our results without any quantitative bounds on the constants appearing in the coupling time and total variation estimates, it is possible to track concrete numerical bounds from the proofs presented above.
\end{rem}
We need the following elementary fact. For any $x\geqslant 0$ and $0\leqslant y \leqslant 1$
\begin{equation}\label{e.3.2}
x+y\leqslant \sqrt{2}\left(x^{2}+y\right)^{\frac{1}{2}}.
\end{equation}
Indeed,

\begin{equation*}
\left(x+y\right)^{2} \leqslant 2x^{2}+2y^{2} \leqslant 2\left(x^{2}+y\right),
\end{equation*}
since $y \leqslant 1$. This immediately gives us the following result.

\begin{prop}\label{prop:ccc}
Assume that $\left|a-\widetilde{a}+b_{1}\widetilde{b}_{2}-b_{2}\widetilde{b}_{1}\right|<1$. Then there exists a constant $C>0$ such that
\[
\mathbb{P}\left(\tau>t\right)\leqslant \frac{C}{\sqrt{t}}d_{CC}\left(\left(b_{1},b_{2},a\right),\left(\widetilde{b_{1},}\widetilde{b_{2}},\widetilde{a}\right)\right)
\]
for $t\geqslant \max\left\{ \left|\mathbf{b}-\widetilde{\mathbf{b}}\right|^{2}, 2\left|a-\widetilde{a}+b_{1}\widetilde{b}_{2}-b_{2}\widetilde{b}_{1}\right|, 1 \right\}$.

\end{prop}

\begin{proof}
Since $t>1$, then $\frac{1}{t}\leqslant \frac{1}{\sqrt{t}}$, so by Theorem \ref{thm:3}
\begin{align*}
\mathbb{P}\left(\tau>t\right) & \leqslant  C\left(\frac{\left|\mathbf{b}- \widetilde{\mathbf{b}}\right|}{\sqrt{t}}+\frac{\left|a- \widetilde{a}+b_{1}\widetilde{b}_{2}-b_{2}\widetilde{b}_{1}\right|}{t}\right)
\\
 & \leqslant  \frac{C}{\sqrt{t}}\left(\left|\mathbf{b}-\widetilde{\mathbf{b}}\right|+\left|a-\widetilde{a}+b_{1}\widetilde{b}_{2}-b_{2}\widetilde{b}_{1}\right|\right)\\
 & \leqslant  \frac{C}{\sqrt{t}}\left(\left|\mathbf{b}-\widetilde{\mathbf{b}}\right|^{2}+\left|a-\widetilde{a}+b_{1}\widetilde{b}_{2}-b_{2}\widetilde{b}_{1}\right|\right)^{\frac{1}{2}}
\end{align*}
where we used \eqref{e.3.2} in the last inequality. Now we consider

\[
\rho\left(\left(b_{1}, b_{2}, a\right), \left(\widetilde{b_{1},}, \widetilde{b_{2}},\widetilde{a}\right)\right)=
\left(\left|\mathbf{b}-\widetilde{\mathbf{b}}\right|^{2} +\left|a-\widetilde{a}+b_{1}\widetilde{b}_{2}-b_{2}\widetilde{b}_{1}\right|\right)^{\frac{1}{2}},
\]
as defined by \eqref{e.2.1}. Recall from Section \ref{section:prelim} that this pseudo-metric is equivalent to the Carnot-Carath\'{e}odory distance $d_{CC}\left( \left(b_{1}, b_{2}, a\right), \left(\widetilde{b_{1}}, \widetilde{b_{2}}, \widetilde{a}\right) \right)$. This gives us the desired inequality.
\end{proof}

Liouville type theorems have been known for the Heisenberg group and other types of Carnot groups (e.g. \cite[Theorem 5.8.1]{BonfiglioliLanconelliUguzzoniBook}). Using the coupling we constructed, we derive a functional inequality (a form of which appeared as \cite[Equation (24)]{BakryBaudoinBonnefontChafai2008}) which consequently gives us the Liouville property rather easily.

In the following, for any bounded measurable function $u: \mathbb{H}^3 \rightarrow \mathbb{R}$ and any $x \in \mathbb{H}^3$, we define

\[
P_t u(x)=\mathbb{E}u\left(\mathbf{X}^x_{t}\right),
\]
where $\mathbf{X}^x$ is a Brownian motion on the Heisenberg group starting from $x$. By $\Vert \cdot \Vert_{\infty}$ we denote the sup norm. %We let $C_c^\infty(\mathbb{H}^3)$ denote the smooth and compactly supported functions on $\mathbb{H}^3$.
\begin{corollary}
For any bounded $u\in C^\infty(\mathbb{H}^3)$  there exists a positive constant $C$, which does not depend on $u$, such that for any $t  \geqslant 1$
\begin{align}\label{align:lio}
\Vert \nabla_{\mathcal{H}} P_t u \Vert_{\infty} \leqslant \frac{C}{\sqrt{t}}\Vert u \Vert_{\infty}.
\end{align}
Consequently, if $\Delta_{\mathcal{H}}u=0$, then $u$ is a constant.
\end{corollary}

\begin{proof}
Fix $t \geqslant 1$. Take two distinct points $\left(b_{1}, b_{2},a\right)$ and $\left(\widetilde{b_{1}}, \widetilde{b_{2}}, \widetilde{a}\right)$ in $\left( \mathbb{H}^3, d_{CC} \right)$ sufficiently close to $\left(b_{1}, b_{2}, a\right)$ with respect to the distance $d_{CC}$ in such a way that
\[
\max\left\{ \left|\mathbf{b}-\widetilde{\mathbf{b}}\right|^{2},2\left|a-\widetilde{a}+b_{1}\widetilde{b}_{2}-b_{2}\widetilde{b}_{1}\right|\right\} \leqslant 1.
\]
Then, using the coupling $(\mathbf{X}, \widetilde{\mathbf{X}})$ constructed in Theorem \ref{thm:3} and by Proposition \ref{prop:ccc}, we get
%\max\left\{ \left|\mathbf{b}-\widetilde{\mathbf{b}}\right|^{2},2\left|a-\widetilde{a}+b_{1}\widetilde{b}_{2}-b_{2}\widetilde{b}_{1}\right|,1\right\} $,
\begin{align*}
\left|P_tu\left(b_{1}, b_{2}, a\right)-P_tu\left(\widetilde{b_{1}}, \widetilde{b_{2}}, \widetilde{a}\right)\right|&=
\left|\mathbb{E}\left(u\left(\mathbf{X}_{t}\right) - u\left(\widetilde{\mathbf{X}}_{t}\right):\tau>t\right)\right|
\\
\leqslant  2\left\Vert u\right\Vert _{\infty}\mathbb{P}\left(\tau>t\right)
& \leqslant \frac{2C}{\sqrt{t}}\left\Vert u\right\Vert _{\infty}d_{CC}\left(\left(b_{1}, b_{2},a\right),\left(\widetilde{b_{1},}\widetilde{b_{2}},\widetilde{a}\right)\right).
\end{align*}
Dividing by $d_{CC}\left(\left(b_{1},b_{2},a\right),\left(\widetilde{b_{1},}\widetilde{b_{2}},\widetilde{a}\right)\right)$ on both sides above and taking a supremum over all points $\left(\widetilde{b_{1},}\widetilde{b_{2}},\widetilde{a}\right)
\neq \left(b_{1}, b_{2}, a\right)$, we get \eqref{align:lio}.

Finally if $\Delta_{\mathcal{H}}u=0$, then $P_tu=u$ for all $t \geqslant 0$. Taking $t\to\infty$ in \eqref{align:lio}, we get $\nabla_{\mathcal{H}}u \equiv 0$ and hence $u\in C^\infty(\mathbb{H}^3)$ is constant by  \cite[Proposition 1.5.6]{BonfiglioliLanconelliUguzzoniBook}.

\end{proof}

\section{Gradient estimates}\label{section:Gradient estimates}

The goal of this section is to prove gradient estimates using the coupling construction introduced earlier. Let $x=(b_{1},b_{2},a)$ and $\widetilde{x}=(\widetilde{b_{1}},\widetilde{b_{2}}, \widetilde{a})$. We let $(\mathbf{X},\widetilde{\mathbf{X}})$ be the non-Markovian coupling of two Brownian motions $\mathbf{X}$ and $\widetilde{\mathbf{X}}$ on the Heisenberg group starting from $x$ and $\widetilde{x}$ respectively as described in Theorem \ref{thm:3}. For a set $Q$, define the exit time of a process $\mathbf{X}_{t}$ from this set by
\[
\tau_{Q}\left(\textbf{X}\right)=\inf\left\{ t>0:\mathbf{X}_{t}\notin Q\right\}.
\]
The oscillation of a function over a set $Q$ is defined by
\[
\operatorname{osc}_{Q}u\equiv\sup_{Q}u-\inf_{Q}u.
\]

Before we can formulate and prove the main results of this section, Theorems \ref{thm:4} and \ref{t.GradientEst}, we need two preliminary results. Lemma \ref{lemma:BDG} gives second moment estimates for $\sup_{t \leqslant \tau \wedge 1}|\int_0^t(B_2(s)-b_2)dB_1(s)|$, $\sup_{t \leqslant \tau \wedge 1} |B_1(t)-b_1|$ and $\sup_{t \leqslant \tau \wedge 1} |B_2(t)-b_2|$ under the coupling constructed above, when the coupled Brownian motions start from the same point $(b_1,b_2)$. It would be natural to want to apply here Burkholder-Davis-Gundy (BDG) inequalities such as \cite[p. 163]{KaratzasShreveBMBook}) which give sharp estimates of moments of $\sup_{t \leqslant T} |M_t|$ for any continuous local martingale $M$ in terms of the moments of its quadratic variation $\langle M \rangle_T$ when $T$ is a stopping time. But the coupling time $\tau$ is \textit{not a stopping time} with respect to the filtration generated by $(B_1, B_2)$, and therefore we \textit{can not apply} these inequalities to get the moment estimates.
\begin{lemma}\label{lemma:BDG}
Consider the coupling of the diffusions

\begin{align*}
\left\{ \left(B_{1}(t), B_{2}(t), a+\int_0^t B_{2}(s)dB_{1}(s)\right): t \geqslant 0 \right\}
\\ \left\lbrace\left( \widetilde{B}_{1}(t), \widetilde{B}_{2}(t), \widetilde{a}+\int_0^t\widetilde{B}_{2}(s)d\widetilde{B}_{1}(s)\right): t \geqslant 0\right\rbrace,
\end{align*}
described in Lemma \ref{lem:1}, with $B_1(0)=\widetilde{B}_1(0)=b_{1}$, $B_2(0)=\widetilde{B}_2(0)=b_2$ and $a >\widetilde{a}$, with coupling time $\tau$. Then there exists a positive constant $C$ not depending on $b_1, b_2, a, \widetilde{a}$ such that we have the following
\begin{itemize}
\item[(i)] $\mathbb{E}\left(\sup_{t \leqslant \tau \wedge 1} \left|\int_0^t (B_2(s)-b_2) d B_1(s)\right|\right)^2 \leqslant C\mathbb{E}(\tau \wedge 1)^2,$
\item[(ii)] $\mathbb{E}\left(\sup_{t \leqslant \tau \wedge 1} |B_1(t)-b_1|\right)^4 \leqslant C \mathbb{E}(\tau \wedge 1)^2,$
\item[(iii)] $\mathbb{E}\left(\sup_{t \leqslant \tau \wedge 1} |B_2(t)-b_2|\right)^4 \leqslant C \mathbb{E}(\tau \wedge 1)^2.$
\end{itemize}
\end{lemma}
\begin{proof}
In this proof, $C$ will denote a generic positive constant whose value does not depend on $b_1, b_2, a, \widetilde{a}$. Our basic strategy will be to find appropriate enlargements of the natural filtration generated by $(B_1, B_2)$ under which $\tau$ becomes a stopping time, and then use the Burkholder-Davis-Gundy inequality.

It suffices to prove the statement for $b_1=b_2=0$. Moreover, using scaling of Brownian motion, it is straightforward to check that it is sufficient to prove the statement with $a-\widetilde{a}=1$ and $\tau \wedge 1$ replaced by $\tau \wedge M$ (for arbitrary $M>0$). We write $B_2(t)= Y_1(t) + Y_2(t)$, where
\begin{align}\label{align:Yeq}
Y_1(t)&=\sum_{n=0}^{\infty} 2^{n/2}Z_1^{(n)} g_{n,1}((t-2^n+1)^+ \wedge 2^n)\nonumber\\
Y_2(t)&=\sum_{n=0}^{\infty}2^{n/2}\left(\frac{(t-2^n+1)^+ \wedge 2^n}{2^n}Z_0^{(n)} + \sum_{k=2}^{\infty}Z_k^{(n)} g_{n,k}((t-2^n+1)^+ \wedge 2^n)\right)
\end{align}
with $g_{n, k}(t)=g_{2^{n}, k}\left( t \right)
%(\sqrt{2}/(k\pi))\sin(k\pi t/2^n)
$ as defined in the Karhunen-Lo\`{e}ve expansion \eqref{eq:2} and $Z_0^{(n)}= 2^{-n/2}G^{(2^n)}$ for a a Gaussian variable with mean zero and variance $2^{n}$ as we used in \eqref{eq:1}.

Consider the filtration
\begin{align*}
\mathcal{F}^*_t = \sigma \left(\{B_1(s): s \leqslant t\} \cup \{W^{(n)}(s): n \geqslant 0, 0 \leqslant s \leqslant \infty\} \cup \{Z^{(n)}_k: n \geqslant 0, k \geqslant 2\} \right).
\end{align*}
We assume without loss of generality that $\{\mathcal{F}^*_t\}_{t \geqslant  0}$ is augmented, in the sense that all the null sets of $\mathcal{F}^*_{\infty}$ and their subsets lie in $\mathcal{F}^*_0$. We claim that $\tau$ is a stopping time under the above filtration. To see this, recall that by the definition of coupling time, the coupled processes must evolve together after the coupling time and thus, by the coupling construction given in Lemma \ref{lem:1} (in particular, see \eqref{levypart}),
\begin{equation}\label{taudiscrete}
\mathbb{P}[\tau \in \{2^{n+1}-1: n \geqslant  0\}]=1.
\end{equation}
Thus, to show that $\tau$ is a stopping time with respect to $\mathcal{F}^*_t $, it suffices to show that $\{\tau > 2^{n+1}-1\}$ is measurable with respect to $\mathcal{F}^*_{2^{n+1}-1}$ for each $n \geqslant  0$. This is because, for $t \in [2^{n+1}-1, 2^{n+2}-1)$ ($n \geqslant  0$),
\[
\{\tau > t\} = \{\tau > 2^{n+1}-1\}
\]
almost surely with respect to the coupling measure $\mathbb{P}$, by \eqref{taudiscrete}.
Note that for any $n \geqslant 0$,
\[
\{\tau > 2^{n+1}-1\} = \bigcap_{m=0}^n \{\sigma^{(m)} > 2^m\}.
\]
Recall that
\begin{align*}
\sigma^{(m)}& =\inf\left\{ t\geqslant 0:W^{(m)}(t)=
\right.
\\
& \left.
-\left(I(2^{m}-1)-\widetilde{I}(2^{m}-1)\right)\Bigg/\left(2\int_{2^{m}-1}^{2^{m+1}-1}g_{m,1}(s-2^m+1)dB_{1}(s)\right)\right\}
\end{align*}
and on the event $\{\tau > 2^{m+1}-1\}$,
\[
B_2(s) - \widetilde{B}_2(s)= Y_1(s) - \widetilde{Y}_1(s) = 2Y_1(s), \ \ \text{ for all } 0 \leqslant s \leqslant 2^{m+1}-1.
\]
As $\{Y_1(t): 0 \leqslant t \leqslant 2^{m+1}-1\}$ depends measurably on $\{Z_1^{(k)}: 0 \leqslant k \leqslant m\}$ and hence on $\{W^{(k)}(s): k \geqslant 0, 0 \leqslant s < \infty \}$, the above representation for $\sigma^{(m)}$ implies that the event $\{\sigma^{(m)} >2^m\}$ is measurable with respect to $\mathcal{F}^*_{2^{m+1}-1}$. Thus, for each $n \geqslant 0$, $\{\tau > 2^{n+1}-1\}$ is measurable with respect to $\mathcal{F}^*_{2^{n+1}-1}$ and hence, $\tau$ is indeed a stopping time with respect to $\{\mathcal{F}^*_t\}_{t \geqslant 0}$.

Also, note that $\left(\int_0^t B_2(s) dB_1(s)\right)_{t \geqslant 0}$ remains a continuous martingale under this enlarged filtration. Thus, by the Burkholder-Davis-Gundy inequality, we get
\begin{align*}
\mathbb{E}\left(\sup_{t \leqslant \tau \wedge M} \left| \int_0^t B_2(s) d B_1(s)\right|\right)^2 \leqslant C \mathbb{E}\left(\int_0^{\tau \wedge M}B_2^2(s)ds\right) \leqslant
\\
C \mathbb{E}\left(\left(\sup_{t \leqslant \tau \wedge M} |B_2(t)|\right)^2 (\tau \wedge M)\right)
\end{align*}
Now, by the Cauchy-Schwarz inequality
\[
 \mathbb{E}\left(\left(\sup_{t \leqslant \tau \wedge M} |B_2(t)|\right)^2 (\tau \wedge M)\right) \leqslant \left(\mathbb{E}\left(\sup_{t \leqslant \tau \wedge M} |B_2(t)|\right)^4 \right)^{1/2} \left(\mathbb{E} (\tau \wedge M)^2\right)^{1/2}.
\]
Thus, to complete the proof (i) and (iii), it suffices to show that
\[
\mathbb{E}\left(\sup_{t \leqslant \tau \wedge M} |B_2(t)|\right)^4 \leqslant C \mathbb{E} (\tau \wedge M)^2.
\]
To show this, define the Brownian motion
\[
W(t)=\sum_{n=0}^{\infty} W^{(n)}\left((t-2^n+1)^+ \wedge 2^n\right)
\]
and the following (augmented) filtration
\[
\mathcal{F}^{**}_t= \sigma\left(\{(B_1(s), W(s)): s \leqslant t\} \cup \{Z^{(n)}_k: n \geqslant 0, k \geqslant 2\} \right).
\]
Exactly as before, we can check that $\tau$ is a stopping time with respect to this new filtration and $W$ is a Brownian motion (hence a continuous martingale) under it. From the representation \eqref{align:Yeq}, note that
\begin{align*}
\sup_{t \leqslant \tau \wedge M} |Y_1(t)| = \frac{\sqrt{2}}{\pi} \sup_{n: 2^{n+1}-1 \leqslant \tau \wedge M} | W(2^{n+1}-1)-W(2^n-1)| \leqslant \frac{2\sqrt{2}}{\pi} \sup_{t \leqslant \tau \wedge M} |W(t)|.
\end{align*}
Thus, by the the Burkholder-Davis-Gundy inequality
\begin{align}\label{align:Yone}
\mathbb{E}\left(\sup_{t \leqslant \tau \wedge M} |Y_1(t)|\right)^4 \leqslant \frac{64}{\pi^4}\mathbb{E}\left(\sup_{t \leqslant \tau \wedge M} |W(t)|\right)^4 \leqslant C \mathbb{E}(\tau \wedge M)^2.
\end{align}
To estimate $\sup_{t \leqslant \tau \wedge M} |Y_2(t)|$, note that $Y_2$ and $\tau$ are independent. Thus, by a conditioning argument, it suffices to show that for fixed $T>0$,
\begin{align}\label{align:Ytwo}
\mathbb{E}\left(\sup_{t \leqslant T} |Y_2(t)|\right)^4 \leqslant CT^2.
\end{align}
To see this, observe that $Y_2(t)=B_2(t)-Y_1(t)$ for each $t \geqslant 0$ and thus
\[
\sup_{t \leqslant T}|Y_2(t)| \leqslant \sup_{t \leqslant T} |B_2(t)| + \sup_{t \leqslant T} |Y_1(t)|.
\]
Again by the the Burkholder-Davis-Gundy inequality
\[
\mathbb{E}\left(\sup_{t \leqslant T} |B_2(t)|\right)^4 \leqslant CT^2.
\]
By exactly the same argument as the one used to estimate the supremum of $Y_1$, but now applied to a fixed time $T$, we get
\[
\mathbb{E}\left(\sup_{t \leqslant T} |Y_1(t)|\right)^4 \leqslant CT^2.
\]
The two estimates above  yield \eqref{align:Ytwo}, and hence complete the proof of (i) and (iii).

Similarly, (ii) follows from the fact that $B_1$ is a Brownian motion under the filtration $\{\mathcal{F}^*_t\}_{t \geqslant 0}$ and the  Burkholder-Davis-Gundy inequality.
\end{proof}
The next lemma estimates $\mathbb{E}(\tau \wedge 1)^2$.
\begin{lemma}\label{lemma:momest}
Under the coupling of Lemma \ref{lem:1}, there exists a positive constant $C$ not depending on $b_1, b_2, a, \widetilde{a}$ such that

\[
\mathbb{E}(\tau \wedge 1)^2 \leqslant C(|a-\widetilde{a}| \wedge 1).
\]
\end{lemma}
\begin{proof}
Without loss of generality, we assume $|a-\widetilde{a}| \leqslant 1$. We can write
\begin{align*}
\mathbb{E}(\tau \wedge 1)^2 &= \int_0^1\mathbb{P}(\tau > \sqrt{t}) dt\\
&\leqslant |a-\widetilde{a}|^2 + \int_{|a-\widetilde{a}|^2}^1\mathbb{P}(\tau > \sqrt{t}) dt.
\end{align*}
From Lemma \ref{lem:1}, we get a constant $C$ that does not depend on $b_1, b_2, a, \widetilde{a}$ such that for $t > |a-\widetilde{a}|^2$,
\[
\mathbb{P}(\tau > \sqrt{t}) \leqslant C\frac{|a-\widetilde{a}|}{\sqrt{t}}.
\]
Using this we get
\begin{align*}
\mathbb{E}(\tau \wedge 1)^2 \leqslant |a-\widetilde{a}|^2 + C|a-\widetilde{a}|\int_0^1\frac{1}{\sqrt{t}}dt \leqslant (1+2C)|a-\widetilde{a}|,
\end{align*}
which proves the lemma.
\end{proof}

Let $D \subset \mathbb{H}^3$ be a domain. Later in Theorem \ref{t.GradientEst} we give gradient estimates for harmonic functions in $D$, but we start by a result on the coupling time $\tau$.
Define the Heisenberg ball of radius $r>0$ with respect to the distance $\rho$
\[
B(x, r)= \{y \in \mathbb{H}^3: \rho(x,y) < r\}.
\]
Recall that $\rho$ is the pseudo-metric equivalent to $d_{CC}$  defined by \eqref{e.2.1}. For $x\in D$, let $\delta_{x}=\rho\left(x, D^c\right)$.

Consider the coupling of two Brownian motions on the Heisenberg group $\mathbf{X}$ and $\widetilde{\mathbf{X}}$ starting from points $x, \widetilde{x} \in D$ respectively as described by Theorem \ref{thm:3}. We choose these points in such a way that $\rho(x,\widetilde{x})$ is small enough compared to $\delta_x$. The following theorem estimates the probability (as a function of $\delta_x$ and $\rho(x,\widetilde{x})$) that one of the processes exits the ball $B(x,\delta_x)$ before coupling happens. This turns out to be pivotal in proving the gradient estimate.

\begin{theorem}\label{thm:4}
Let $x=(b_1, b_2, a) \in D$, $\widetilde{x}=(\widetilde{b}_1, \widetilde{b}_2, \widetilde{a}) \in D$ such that $\rho(x, \widetilde{x}) < \delta_x/32$, $|\mathbf{b}-\widetilde{\mathbf{b}}| \leqslant 1$ and $|a-\widetilde{a} +b_1\widetilde{b}_2 -b_2\widetilde{b}_1|\leqslant 1/2$.
%\[
%Q=\left\{ \left(y_{1},y_{2},y_{3}\right)\in\mathbb{R}^{3} : \max_{i=1,2}\left|y_{i}-b_i\right|<\frac{\delta_{x}}{4},\left|a-y_{3}+b_1y_{2}-b_2y_{1}\right|<\frac{\delta_{x}^2}{4}\right\}
%\]
Then, under the same coupling of Theorem $\ref{thm:3}$, there exists a constant $C>0$ that does not depend on $x, \widetilde{x}$ such that
\[
\mathbb{P}\left(\tau>\tau_{B(x, \delta_x)}\left(\mathbf{X}\right)\wedge\widetilde{\tau}_{B(x, \delta_x)}\left(\widetilde{\mathbf{X}}\right)\right)\leqslant C\left(1+\frac{1}{\delta_x} + \frac{1}{\delta^4_x} + \frac{(1+\delta_x)^3}{\delta^4_x}\right)\rho(x,\widetilde{x}).
\]
\end{theorem}

\begin{proof}
%It is routine to check that if $y \in Q$, then $\rho(x,y) < \delta_x$ and thus, $Q \subset D$.
In this proof, $C$ will denote a generic positive constant (whose value might change from line to line) that does not depend on $x, \widetilde{x}$.

Let $\hat{b}_i = \frac{b_i + \widetilde{b}_i}{2}$ for $i=1,2$ and $\hat{a} = \frac{a + \widetilde{a}}{2}$. We define the Heisenberg cube by
\[
Q=\left\{ \left(y_{1},y_{2},y_{3}\right)\in\mathbb{R}^{3} : \max_{i=1,2}\left|y_{i}-\hat{b}_i\right| \le\frac{\delta_{x}}{8},\left|\hat{a}-y_{3}+\hat{b}_1y_{2}-\hat{b}_2y_{1}\right| \leqslant \frac{\delta_{x}^2}{16}\right\}.
\]
Write $\hat{x}=(\hat{b}_1, \hat{b}_2, \hat{a})$. It is straightforward to check that $\rho(x, \hat{x}) \leqslant \rho(x, \widetilde{x})/\sqrt{2} < \delta_x/32\sqrt{2}$. Moreover, for $y \in Q$
\begin{align*}
\rho(\hat{x},y) & = \left(|y_1-\hat{b}_1|^2 + |y_2-\hat{b}_2|^2 + \left|\hat{a}-y_{3}+\hat{b}_1y_{2}-\hat{b}_2y_{1}\right|\right)^{1/2}\\
& \leqslant |y_1-\hat{b}_1| +|y_2-\hat{b}_2| + \left|\hat{a}-y_{3}+\hat{b}_1y_{2}-\hat{b}_2y_{1}\right|^{1/2} \leqslant \delta_x/2.
\end{align*}
Thus, by the triangle inequality, for any $y \in Q$
\[
\rho(x,y) \leqslant \rho(x,\hat{x}) + \rho(\hat{x},y) < \delta_x
\]
and hence, $Q \subset B(x, \delta_x)$.
Note that we can write $Q = Q_1 \cap Q_2$ where
\begin{eqnarray*}
Q_{1} & = & \left\{ \left(y_{1},y_{2}, y_3\right)\in\mathbb{R}^{3} : \max_{i=1,2}\left|y_{i}-\hat{b}_i\right| \le\frac{\delta_{x}}{8}\right\} ,\\
Q_{2} & = & \left\{ (y_1, y_2, y_3)\in\mathbb{R}^3 : \left|\hat{a}-y_{3}+\hat{b}_1y_{2}-\hat{b}_2y_{1}\right| \leqslant \frac{\delta_{x}^2}{16}\right\}.
\end{eqnarray*}
As the L\'{e}vy stochastic area is invariant under rotations of
coordinates, it suffices to assume that $b_{1}=\widetilde{b}_{1}$. We define
\[
U(t)= a-\hat{a} + \int_0^tB_1(s)dB_2(s) - \int_0^tB_2(s)dB_1(s) +B_1(t)\hat{b}_2 - B_2(t) \hat{b}_1.
\]
%and similarly $\widetilde{U}$ in terms of $(B_1, \widetilde{B}_2)$.
Note that
\[
d U(t)=(B_1(t)-\hat{b}_1)dB_2(t) - (B_2(t)-\hat{b}_2)dB_1(t).
\]
Writing
\[
\sigma_u=\inf\{t \geqslant 0: |U(t)| > u\},
\]
we observe that $\tau_{Q_2}(\mathbf{X}) = \sigma_{\delta_x^2/16}$ and hence, $\tau_{Q}\left(\mathbf{X}\right)=\tau_{Q_{1}}\left(\mathbf{X}\right)\wedge\tau_{Q_{2}}\left(\mathbf{X}\right) = \tau_{Q_{1}}\left(\mathbf{X}\right)\wedge\sigma_{\delta_x^2/16}$. We can write
\begin{align*}
\mathbb{P}\left(\tau>\tau_{B(x, \delta_x)}\left(\mathbf{X}\right)\wedge\widetilde{\tau}_{B(x, \delta_x)}\left(\widetilde{\mathbf{X}}\right)\right) &\leqslant \mathbb{P}(\tau>\tau_Q(\mathbf{X}) \wedge \tau_Q(\widetilde{\mathbf{X}}))\\
& \leqslant \mathbb{P}(\tau>\tau_Q(\mathbf{X})) +  \mathbb{P}(\tau>\tau_Q(\widetilde{\mathbf{X}})).
\end{align*}
Now we estimate $\mathbb{P}(\tau>\tau_Q(\mathbf{X}))$, the second term in the inequality above can be estimated similarly. First we define
\[
Q^*_{1}  = \left\{ \left(y_{1},y_{2}, y_3\right)\in\mathbb{R}^{3} : \max_{i=1,2}\left|y_{i}-\hat{b}_i\right| \leqslant \frac{\delta_{x}}{16}\right\}.
\]
We have
\begin{align}\label{align:exp}
\mathbb{P}(\tau>\tau_Q(\mathbf{X})) &= \mathbb{P}(\tau>\tau_{Q_1}(\mathbf{X}) \wedge \sigma_{\delta_x^2/16})\nonumber\\
& \leqslant \mathbb{P}(T_1 > \tau_{Q^*_1}(\mathbf{X})) + \mathbb{P}(\tau>\tau_{Q_1}(\mathbf{X}) \wedge \sigma_{\delta_x^2/16}, T_1 \leqslant \tau_{Q^*_1}(\mathbf{X}))\nonumber\\
&\leqslant \mathbb{P}(T_1 > \tau_{Q^*_1}(\mathbf{X})) + \mathbb{P}(\sigma_{\delta_x^2/32} \leqslant T_1 \wedge \tau_{Q^*_1}(\mathbf{X}))\nonumber\\
& \quad \quad + \mathbb{P}(\tau>\tau_{Q_1}(\mathbf{X}) \wedge \sigma_{\delta_x^2/16}, T_1 \leqslant \tau_{Q^*_1}(\mathbf{X}) \wedge \sigma_{\delta_x^2/32}).
\end{align}
It follows from a computation involving standard Brownian estimates (see, for example, the proof of \cite[Theorem 1]{Cranston1992a}) that
\begin{align}\label{align:e1}
\mathbb{P}(T_1 > \tau_{Q^*_1}(\mathbf{X})) \leqslant C\frac{|\mathbf{b}-\widetilde{\mathbf{b}}|}{\delta_x}.
\end{align}
To estimate the second term in \eqref{align:exp}, note that
\[
\mathbb{P}(\sigma_{\delta_x^2/32} \leqslant T_1 \wedge \tau_{Q^*_1}(\mathbf{X}))= \mathbb{P}\left(\sup_{t \leqslant T_1 \wedge \tau_{Q^*_1}(\mathbf{X})}|U(t)|> \frac{\delta_x^2}{32}\right).
\]
Now, as $T_1 \wedge \tau_{Q^*_1}(\mathbf{X})$ is a stopping time with respect to the natural filtration generated by $(B_1,B_2)$, by the the Burkholder-Davis-Gundy inequality
\begin{align*}
\mathbb{E}&\left(\sup_{t \leqslant T_1 \wedge \tau_{Q^*_1}(\mathbf{X})}|U(t)-U(0)|\right)^2\\
& \leqslant C \mathbb{E}\left(\int_0^{T_1 \wedge \tau_{Q^*_1}(\mathbf{X})}|\mathbf{B}(s)-\hat{\mathbf{b}}|^2ds\right)\\
&\leqslant C\mathbb{E}\left(\int_0^{T_1 \wedge \tau_{Q^*_1}(\mathbf{X})}\delta^2_xds\right)\\
&\leqslant C\delta^2_x\mathbb{E}(T_1 \wedge \tau_{Q^*_1}(\mathbf{X})).
%&\leqslant C\sqrt{\mathbb{E}\left(\sup_{t \leqslant T_1 \wedge \tau_{Q^*_1}(\mathbf{X})}|\mathbf{B}(t)-\mathbf{b}|^4 + |\mathbf{b}-\hat{\mathbf{b}}|^4\right)}\sqrt{\mathbb{E} \left(T_1 \wedge \tau_{Q^*_1}(\mathbf{X}))\right)^2}\\
%&\leqslant C\sqrt{\mathbb{E} \left(T_1 \wedge \tau_{Q^*_1}(\mathbf{X}))\right)^2 + |\mathbf{b}-\hat{\mathbf{b}}|^4}\sqrt{\mathbb{E} \left(T_1 \wedge \tau_{Q^*_1}(\mathbf{X}))\right)^2} \ \ \text{(by BDG inequality)}\\
%&\leqslant C\mathbb{E} \left(T_1 \wedge \tau_{Q^*_1}(\mathbf{X}))\right)^2 + |\mathbf{b}-\hat{\mathbf{b}}|^2\sqrt{\mathbb{E} \left(T_1 \wedge \tau_{Q^*_1}(\mathbf{X}))\right)^2}
\end{align*}
We can again appeal to standard Brownian estimates  (e.g. see the proof of \cite[Theorem 1]{Cranston1992a}) to see that
\begin{equation}\label{momstop}
\mathbb{E} \left(T_1 \wedge \tau_{Q^*_1}(\mathbf{X}))\right) \leqslant C \delta_x |\mathbf{b}-\hat{\mathbf{b}}|.
\end{equation}
Using this estimate  gives us
\begin{align*}
& \mathbb{E}\left(\sup_{t \leqslant T_1 \wedge \tau_{Q^*_1}(\mathbf{X})}|U(t)|\right)^2  \leqslant
2 \mathbb{E}\left(\sup_{t \leqslant T_1 \wedge \tau_{Q^*_1}(\mathbf{X})}|U(t)-U(0)|\right)^2 +2|U(0)|^2
\\
& \leqslant C\delta^3_x|\mathbf{b}-\hat{\mathbf{b}}| + 2|a-\hat{a} +b_1\hat{b}_2 -b_2\hat{b}_1|^2 \leqslant
 \frac{C}{2}\delta^3_x|\mathbf{b}-\widetilde{\mathbf{b}}| + \frac{1}{2}|a-\widetilde{a} +b_1\widetilde{b}_2 -b_2\widetilde{b}_1|^2.
\end{align*}
By assumption $|a-\widetilde{a} +b_1\widetilde{b}_2 -b_2\widetilde{b}_1| <1$, and  therefore
\begin{align*}
&\mathbb{E}\left(\sup_{t \leqslant T_1 \wedge \tau_{Q^*_1}(\mathbf{X})}|U(t)|\right)^2 \leqslant
C(1+\delta_x)^3(|\mathbf{b}-\widetilde{\mathbf{b}}| + |a-\widetilde{a} +b_1\widetilde{b}_2 -b_2\widetilde{b}_1|)
\\
& \leqslant C(1+\delta_x)^3\rho(x,\widetilde{x}),
\end{align*}
where the last inequality follows from \eqref{e.3.2}. Thus, by the Chebyshev inequality
\begin{align*}
\mathbb{P}\left(\sup_{t \leqslant T_1 \wedge \tau_{Q^*_1}(\mathbf{X})}|U(t)|> \frac{\delta_x^2}{32}\right) \leqslant C\frac{(1+\delta_x)^3}{\delta^4_x}\rho(x,\widetilde{x}),
\end{align*}
which, in turn, gives us
\begin{align}\label{align:e2}
\mathbb{P}(\sigma_{\delta_x^2/32} \leqslant T_1 \wedge \tau_{Q^*_1}(\mathbf{X})) \leqslant C\frac{(1+\delta_x)^3}{\delta^4_x}\rho(x,\widetilde{x}).
\end{align}
To estimate the last term in \eqref{align:exp}, we write
\begin{align}\label{align:f0}
& \mathbb{P}(\tau>\tau_{Q_1}(\mathbf{X}) \wedge \sigma_{\delta_x^2/16}, T_1 \leqslant \tau_{Q^*_1}(\mathbf{X}) \wedge \sigma_{\delta_x^2/32}) \leqslant \mathbb{P}(\tau-T_1>1)\notag\\
& \quad + \mathbb{P}(\tau>\tau_{Q_1}(\mathbf{X}) \wedge \sigma_{\delta_x^2/16}, T_1 \leqslant \tau_{Q^*_1}(\mathbf{X}) \wedge \sigma_{\delta_x^2/32}, \tau-T_1 \leqslant 1).
\end{align}
By Lemma \ref{lem:1}, we get
\[
\mathbb{P}(\tau - T_1>1) \leqslant C \mathbb{E}|A(T_1) \wedge 1|,
\]
where $A$ is the invariant difference of stochastic areas defined in \eqref{align:inv}.
%It is easy to see from the definition of $U$ that for $t \leqslant T_1$, $U(t) + \widetilde{U}(t)=0$.

Applying Lemma \ref{invcontrol} with $t=1$ and appealing to our assumption that $|\mathbf{b}-\widetilde{\mathbf{b}}| \leqslant 1$ and $|a-\widetilde{a} +b_1\widetilde{b}_2 -b_2b_1|\leqslant 1/2$, we have
\[
\mathbb{E}|A(T_1) \wedge 1| \leqslant C(|\mathbf{b}-\widetilde{\mathbf{b}}| + |a-\widetilde{a} +b_1\widetilde{b}_2 -b_2b_1|) \leqslant C \rho(x, \widetilde{x}).
\]
which gives
\begin{align}\label{align:f1}
\mathbb{P}(\tau - T_1>1) \leqslant C \rho(x, \widetilde{x}).
\end{align}
Finally, we need to estimate $\mathbb{P}(\tau>\tau_{Q_1}(\mathbf{X}) \wedge \sigma_{\delta_x^2/16}, T_1 \leqslant \tau_{Q^*_1}(\mathbf{X}) \wedge \sigma_{\delta_x^2/32}, \tau-T_1 \leqslant 1)$. Note that
\begin{align}\label{align:exp1}
\mathbb{P}(\tau>& \tau_{Q_1}(\mathbf{X}) \wedge \sigma_{\delta_x^2/16}, T_1 \leqslant \tau_{Q^*_1}(\mathbf{X}) \wedge \sigma_{\delta_x^2/32}, \tau-T_1 \leqslant 1)\notag
\\
& \leqslant \mathbb{P}\left(\sup_{T_1\leqslant t \leqslant T_1 +(\tau-T_1) \wedge 1}|B_1(t)-B_1(T_1)| \geqslant \delta_x/16\right) + \notag
\\
&
\mathbb{P}\left(\sup_{T_1\leqslant t \leqslant T_1 +(\tau-T_1) \wedge 1}|B_2(t)-B_2(T_1)| \geqslant \delta_x/16\right)\nonumber
\\
&\quad + \mathbb{P}\left(\sup_{T_1\leqslant t \leqslant T_1 +(\tau-T_1) \wedge 1} |U(t)-U(T_1)| \geqslant \delta^2_x/32,\right.\nonumber\\
&\qquad \qquad \left. \sup_{T_1\leqslant t \leqslant T_1 +(\tau-T_1) \wedge 1}|B_1(t)-B_1(T_1)| < \delta_x/16, T_1 \leqslant \tau_{Q^*_1}(\mathbf{X})\right).
\end{align}
By the strong Markov property applied at $T_1$, along with parts (ii) and (iii) of Lemma \ref{lemma:BDG} and the Chebyshev inequality, we get
\[
\mathbb{P}\left(\sup_{T_1\leqslant t \leqslant T_1 +(\tau-T_1) \wedge 1}|B_i(t)-B_i(T_1)| \geqslant \delta_x/16\right) \leqslant C\frac{\mathbb{E}((\tau-T_1) \wedge 1)^2}{\delta^4_x}
\]
for $i=1,2$. From the explicit construction of the coupling strategy given in Theorem \ref{thm:3} and Lemma \ref{lemma:momest} and Lemma \ref{invcontrol}, we obtain
\[
\mathbb{E}((\tau-T_1) \wedge 1)^2 \leqslant \mathbb{E}|A(T_1) \wedge 1| \leqslant C\rho(x,\widetilde{x}).
\]
and thus,
\begin{align}\label{align:f2}
\mathbb{P}\left(\sup_{T_1\leqslant t \leqslant T_1 +(\tau-T_1) \wedge 1}|B_i(t)-B_i(T_1)| \geqslant \delta_x/16\right) \leqslant C\frac{\rho(x,\widetilde{x})}{\delta^4_x}.
\end{align}
for $i=1,2$. To handle the last term in \eqref{align:exp1}, define
\[
U^*(t)=U(t)-(B_1(t)-\hat{b}_1)(B_2(t)-\hat{b}_2).
\]
Note that
\[
dU^*(t)=-2(B_2(t)-\hat{b}_2)dB_1(t).
\]
and $U^*(T_1)=U(T_1)$ as $B_2(T_1)=\hat{b}_2$. Further, observe that
\begin{align*}
& \sup_{T_1\leqslant t \leqslant T_1 +(\tau-T_1) \wedge 1}|U(t)-U(T_1)| \leqslant
\\
& \sup_{T_1 \leqslant t \leqslant T_1 +(\tau-T_1) \wedge 1}|U^*(t)-U^*(T_1)| + \sup_{T_1\leqslant t \leqslant T_1 +(\tau-T_1) \wedge 1}|B_1(t)-\hat{b}_1||B_2(t)-\hat{b}_2|.
\end{align*}
Using this, we can bound the last term in \eqref{align:exp1} as
\begin{multline}\label{lastbd}
\mathbb{P}\left(\sup_{T_1\leqslant t \leqslant T_1 +(\tau-T_1) \wedge 1} |U(t)-U(T_1)| \geqslant \delta^2_x/32,\right.\\
\left. \sup_{T_1\leqslant t \leqslant T_1 +(\tau-T_1) \wedge 1}|B_1(t)-B_1(T_1)| < \delta_x/16, T_1 \leqslant \tau_{Q^*_1}(\mathbf{X})\right)\\
\leqslant  \mathbb{P}\left(\sup_{T_1\leqslant t \leqslant T_1 +(\tau-T_1) \wedge 1} |U^*(t)-U^*(T_1)| \geqslant \delta^2_x/64\right)\\
\qquad \qquad + \mathbb{P}\left(\sup_{T_1\leqslant t \leqslant T_1 +(\tau-T_1) \wedge 1}|B_1(t)-\hat{b}_1||B_2(t)-\hat{b}_2| \geqslant \delta^2_x/64,\right.\\
\hspace{4cm} \left. \sup_{T_1\leqslant t \leqslant T_1 +(\tau-T_1) \wedge 1}|B_1(t)-B_1(T_1)| < \delta_x/16, T_1 \leqslant \tau_{Q^*_1}(\mathbf{X})\right).
\end{multline}
By conditioning at time $T_1$ and part (i) of Lemma \ref{lemma:BDG}, followed by applications of Lemma \ref{lemma:momest} and Lemma \ref{invcontrol}, we obtain
\begin{align*}
&\mathbb{E}\left(\sup_{T_1\leqslant t \leqslant T_1 +(\tau-T_1) \wedge 1}|U^*(t)-U^*(T_1)|\right)^2 \leqslant
\\
& 4\mathbb{E}\left(\sup_{T_1\leqslant t \leqslant T_1 +(\tau-T_1) \wedge 1}\left|\int_{T_1}^t(B_2(s)-\hat{b}_2)dB_1(s)\right|\right)^2 \leqslant
\\
& C\mathbb{E}((\tau-T_1) \wedge 1)^2\leqslant \mathbb{E}|A(T_1) \wedge 1| \leqslant C\rho(x,\widetilde{x}).
\end{align*}
Consequently, by the Chebyshev inequality
\begin{align}\label{lastbd1}
\mathbb{P}\left(\sup_{T_1\leqslant t \leqslant T_1 +(\tau-T_1) \wedge 1} |U^*(t)-U^*(T_1)| \geqslant \delta^2_x/64\right) \leqslant C\frac{\rho(x,\widetilde{x})}{\delta^4_x}.
\end{align}
Moreover,
\begin{multline}\label{prod1}
\mathbb{P}\left(\sup_{T_1\leqslant t \leqslant T_1 +(\tau-T_1) \wedge 1}|B_1(t)-\hat{b}_1||B_2(t)-\hat{b}_2| \geqslant \delta^2_x/64,\right.\\
\left. \sup_{T_1\leqslant t \leqslant T_1 +(\tau-T_1) \wedge 1}|B_1(t)-B_1(T_1)| < \delta_x/16, T_1 \leqslant \tau_{Q^*_1}(\mathbf{X})\right)\\
\leqslant \mathbb{P}\left(\sup_{T_1\leqslant t \leqslant T_1 +(\tau-T_1) \wedge 1}|B_2(t)-\hat{b}_2| \geqslant \delta_x/8\right)\\
 + \mathbb{P}\left(\sup_{T_1\leqslant t \leqslant T_1 +(\tau-T_1) \wedge 1}|B_1(t)-\hat{b}_1|\geqslant \delta_x/8,\right.\\
\hspace{4cm} \left. \sup_{T_1\leqslant t \leqslant T_1 +(\tau-T_1) \wedge 1}|B_1(t)-B_1(T_1)| < \delta_x/16, T_1 \leqslant \tau_{Q^*_1}(\mathbf{X})\right).
\end{multline}
We use the fact $B_2(T_1) = \hat{b}_2$ and proceed exactly along the lines of the proof of \eqref{align:f2} to obtain
\begin{equation}\label{inter1}
\mathbb{P}\left(\sup_{T_1\leqslant t \leqslant T_1 +(\tau-T_1) \wedge 1}|B_2(t)-\hat{b}_2| \geqslant \delta_x/8\right) \leqslant  C\frac{\rho(x,\widetilde{x})}{\delta^4_x}.
\end{equation}
The second probability appearing on the right hand side of \eqref{prod1} can be bounded as follows
\begin{multline}\label{inter2}
\mathbb{P}\left(\sup_{T_1\leqslant t \leqslant T_1 +(\tau-T_1) \wedge 1}|B_1(t)-\hat{b}_1|\geqslant \delta_x/8,\right.\\
\left. \sup_{T_1\leqslant t \leqslant T_1 +(\tau-T_1) \wedge 1}|B_1(t)-B_1(T_1)| < \delta_x/16, T_1 \leqslant \tau_{Q^*_1}(\mathbf{X})\right)\\
\leqslant  \mathbb{P}\left(\sup_{(T_1 \wedge  \tau_{Q^*_1}(\mathbf{X}))\leqslant t \leqslant (T_1 \wedge  \tau_{Q^*_1}(\mathbf{X})) +(\tau-(T_1 \wedge  \tau_{Q^*_1}(\mathbf{X}))) \wedge 1}|B_1(t)-\hat{b}_1|\geqslant \delta_x/8,\right.\\
\hspace{2cm}\left. \sup_{(T_1 \wedge  \tau_{Q^*_1}(\mathbf{X}))\leqslant t \leqslant (T_1 \wedge  \tau_{Q^*_1}(\mathbf{X})) +(\tau-(T_1 \wedge  \tau_{Q^*_1}(\mathbf{X}))) \wedge 1}|B_1(t)-B_1(T_1 \wedge  \tau_{Q^*_1}(\mathbf{X}))| < \delta_x/16\right)\\
\leqslant  \mathbb{P}\left(|B_1(T_1 \wedge  \tau_{Q^*_1}(\mathbf{X}))-\hat{b}_1| > \delta_x/16\right).
\end{multline}
We will use the fact that $b_1=\hat{b}_1$. By an application of the Chebyshev inequality followed by the  Burkholder-Davis-Gundy inequality, and using \eqref{momstop}, we get
\begin{multline}\nonumber
\mathbb{P}\left(|B_1(T_1 \wedge  \tau_{Q^*_1}(\mathbf{X}))-\hat{b}_1| > \delta_x/16\right) \leqslant  C\frac{\mathbb{E}|B_{1}(T_{1}\wedge\tau_{Q_{1}^{*}}(\mathbf{X}))-\hat{b}_{1}|^{2}}{\delta_{x}^{2}} \\
 \leqslant  C\frac{\mathbb{E}\sup_{0\leqslant t\leqslant T_{1}\wedge\tau_{Q_{1}^{*}}(\mathbf{X})}\left|B_{1}(t)-b_{1}\right|^{2}}{\delta_{x}^{2}}
\leqslant C\frac{\mathbb{E}(T_1 \wedge  \tau_{Q^*_1}(\mathbf{X}))}{\delta_x^2} \leqslant  C\frac{|\mathbf{b}-\hat{\mathbf{b}}|}{\delta_x}.
\end{multline}
Using this in \eqref{inter2},
 \begin{multline}\label{inter3}
\mathbb{P}\left(\sup_{T_1\leqslant t \leqslant T_1 +(\tau-T_1) \wedge 1}|B_1(t)-\hat{b}_1|\geqslant \delta_x/8,\right.\\
\hspace{2cm} \left. \sup_{T_1\leqslant t \leqslant T_1 +(\tau-T_1) \wedge 1}|B_1(t)-B_1(T_1)| < \delta_x/16, T_1 \leqslant \tau_{Q^*_1}(\mathbf{X})\right) \leqslant  C\frac{|\mathbf{b}-\hat{\mathbf{b}}|}{\delta_x}.
\end{multline}
Using \eqref{inter1} and \eqref{inter3} in \eqref{prod1}, we obtain
\begin{multline}\label{lastbd2}
\mathbb{P}\left(\sup_{T_1\leqslant t \leqslant T_1 +(\tau-T_1) \wedge 1}|B_1(t)-\hat{b}_1||B_2(t)-\hat{b}_2| \geqslant \delta^2_x/64,\right.\\
\left. \sup_{T_1\leqslant t \leqslant T_1 +(\tau-T_1) \wedge 1}|B_1(t)-B_1(T_1)| < \delta_x/16, T_1 \leqslant \tau_{Q^*_1}(\mathbf{X})\right)\\
\leqslant  C\left(\frac{1}{\delta_x} + \frac{1}{\delta^4_x}\right)\rho(x,\widetilde{x}).
\end{multline}
Finally, using \eqref{lastbd1} and \eqref{lastbd2} in \eqref{lastbd},
\begin{multline}\label{f3}
\mathbb{P}\left(\sup_{T_1\leqslant t \leqslant T_1 +(\tau-T_1) \wedge 1} |U(t)-U(T_1)| \geqslant \delta^2_x/32,\right.\\
\left. \sup_{T_1\leqslant t \leqslant T_1 +(\tau-T_1) \wedge 1}|B_1(t)-B_1(T_1)| < \delta_x/16, T_1 \leqslant \tau_{Q^*_1}(\mathbf{X})\right)\\
\leqslant C\left(\frac{1}{\delta_x} + \frac{1}{\delta^4_x}\right)\rho(x,\widetilde{x}).
\end{multline}
%Moreover,
%\begin{align*}
%\mathbb{E}&\left(\sup_{T_1\leqslant t \leqslant T_1 +(\tau-T_1) \wedge 1}|B_1(t)-\hat{b}_1||B_2(t)-\hat{b}_2|\right)^2\\
%&\leqslant \mathbb{E}\left(\sup_{T_1\leqslant t \leqslant T_1 +(\tau-T_1) \wedge 1}|B_1(t)-\hat{b}_1|\right)^2\left(\sup_{T_1\leqslant t \leqslant T_1 +(\tau-T_1) \wedge 1}|B_2(t)-\hat{b}_2|\right)^2\\
%&\leqslant  \left(\mathbb{E}\left(\sup_{T_1\leqslant t \leqslant T_1 +(\tau-T_1) \wedge 1}|B_1(t)-\hat{b}_1|\right)^4\right)^{1/2}\left(\mathbb{E}\left(\sup_{T_1\leqslant t \leqslant T_1 +(\tau-T_1) \wedge 1}|B_2(t)-\hat{b}_2|\right)^4\right)^{1/2}\\
%& \leqslant C\mathbb{E}((\tau-T_1) \wedge 1)^2 \leqslant \mathbb{E}|A(T_1) \wedge 1| \leqslant C\rho(x,\widetilde{x}).
%\end{align*}
%The third inequality above follows from the Cauchy-Schwarz inequality, and the fourth one follows from parts (ii) and (iii) of Lemma \ref{lemma:BDG}.

%Combining the above two estimates, we get
%\[
%\mathbb{E}\left(\sup_{T_1\leqslant t \leqslant T_1 +(\tau-T_1) \wedge 1}|U(t)-U(T_1)|\right)^2 \leqslant C\rho(x,\widetilde{x}).
%\]
%Consequently, by the Chebyshev inequality,
%\begin{align}\label{align:f3}
%\mathbb{P}\left(\sup_{T_1\leqslant t \leqslant T_1 +(\tau-T_1) \wedge 1} |U(t)-U(T_1)| \geqslant \delta^2_x/32\right) \leqslant C\frac{\rho(x,\widetilde{x})}{\delta^4_x}.
%\end{align}
Using the estimates from \eqref{align:f2} and \eqref{f3} in \eqref{align:exp1}, we get
\begin{multline}\label{f4}
\mathbb{P}(\tau> \tau_{Q_1}(\mathbf{X}) \wedge \sigma_{\delta_x^2/16}, T_1 \leqslant \tau_{Q^*_1}(\mathbf{X}) \wedge \sigma_{\delta_x^2/32}, \tau-T_1 \leqslant 1)\\
\leqslant C\left(\frac{1}{\delta_x} + \frac{1}{\delta^4_x}\right)\rho(x,\widetilde{x}).
\end{multline}
Using \eqref{align:f1} and \eqref{f4} in \eqref{align:f0}, we get
\begin{align}\label{align:e3}
\mathbb{P}(\tau>\tau_{Q_1}(\mathbf{X}) \wedge \sigma_{\delta_x^2/16}, T_1 \leqslant \tau_{Q^*_1}(\mathbf{X}) \wedge \sigma_{\delta_x^2/32}) \leqslant C\left(1 + \frac{1}{\delta_x} + \frac{1}{\delta^4_x}\right)\rho(x,\widetilde{x}).
\end{align}
Using the estimates \eqref{align:e1}, \eqref{align:e2} and \eqref{align:e3} in \eqref{align:exp}, we obtain
\begin{align}\label{align:finexp}
\mathbb{P}(\tau>\tau_Q(\mathbf{X})) \leqslant C\left(1+\frac{1}{\delta_x} + \frac{1}{\delta^4_x} + \frac{(1+\delta_x)^3}{\delta^4_x}\right)\rho(x,\widetilde{x}).
\end{align}
The same estimate for  $\mathbb{P}(\tau>\tau_Q(\widetilde{\mathbf{X}}))$ is obtained by interchanging the roles of $x$ and $\widetilde{x}$. This completes the proof of the theorem.
\end{proof}
The above theorem yields the gradient estimate formulated in Theorem \ref{t.GradientEst}. Before we can formulate our result, we explain the argument in the proof of \cite[Proposition 4.1]{Kuwada2010a} that leads to \eqref{gradineq1}.

Recall that $\Delta_{\mathcal{H}}$ denotes the sub-Laplacian which is the generator of the Brownian motion on $\mathbb{H}^3$, and for any function $f$ on $\mathbb{H}^3$, $\left|\nabla_{\mathcal{H}}f\right|$ denotes the associated length of the horizontal gradient of $f$  defined by \eqref{Grad}. As before $\left\Vert \cdot\right\Vert _{\mathcal{H}}$ denotes the norm induced by the sub-Riemannian metric on horizontal vectors. We can use the fact that $\left\{ \mathcal{X}, \mathcal{Y} \right\}$ is an orthonormal frame for the horizontal distribution, therefore for any Lipschitz continuous function $u$ defined on a domain $D$ in $\mathbb{H}^3$,
\[
\left\Vert \nabla_{\mathcal{H}}u\right\Vert_{\mathcal{H}}^{2}=\left(\mathcal{X}u\right)^{2}+\left(\mathcal{Y}u\right)^{2}
\]
holds in $D$ (where $\mathcal{X}u$ and $\mathcal{Y}u$ are interpreted in the distributional sense).
Now we can use \cite[Theorem 11.7]{HajlaszKoskela2000} for the vector fields $\left\{ \mathcal{X}, \mathcal{Y} \right\}$ in $\mathbb{H}^{3}$ identified with $\mathbb{R}^{3}$. We need to check some assumptions in this theorem. First, if $u$ is Lipschitz continuous on $\overline{D}$, it is clear that
\[
\left|\nabla_{\mathcal{H}}u\right|(x) \leqslant \sup_{z, \tilde{z} \in \overline{D}, z\neq\tilde{z}}\frac{\left|u(z)-u\left(\tilde{z}\right)\right|}{d_{CC}\left(z,\tilde{z}\right)}<\infty,
\]
 for all $x \in \overline{D}$, and hence $\left|\nabla_{\mathcal{H}}u\right|$ is locally integrable. In addition, as $u$ is Lipschitz continuous, $\left|\nabla_{\mathcal{H}}u\right|$
is an upper gradient of $u$ by \cite[Lemma 2.1]{Kuwada2010a}, so  \cite[Theorem 11.7]{HajlaszKoskela2000} is applicable and we have that
\begin{equation}\label{gradineq1}
\left\Vert \nabla_{\mathcal{H}}u\right\Vert_{\mathcal{H}}\leqslant\left|\nabla_{\mathcal{H}}u\right|,
\end{equation}
a.e. with respect to the Lebesgue measure.

Let $C\left(\overline{D}\right)$ be the space of functions that are continuous on the closure of the domain $D$. We also let $C^2\left(D\right)$ be the space of functions that are twice  continuously differentiable in $D$.

\begin{thm}\label{t.GradientEst}
Suppose $u\in C\left(\overline{D}\right)\cap C^2\left(D\right)$ such that $\Delta_{\mathcal{H}} u=0$ on $D\subset\mathbb{H}^3$. Fix any constant $\alpha \in (0,1]$. There exists a constant
$C>0$ that does not depend on $u$ such that for every $x \in D$
\begin{equation}\label{gradineq2}
\left\Vert \nabla_{\mathcal{H}}u(x)\right\Vert_{\mathcal{H}} \leqslant \left|\nabla_{\mathcal{H}}u\right|(x)\leqslant C\left(1+\frac{1}{\delta_x} + \frac{1}{\delta^4_x} + \frac{(1+\delta_x)^3}{\delta^4_x}\right)\operatorname{osc}_{B(x,\alpha\delta_x)}u.
\end{equation}
\end{thm}

\begin{proof}

It clearly suffices to consider the case $\alpha=1$. Since $u$ is continuous on $\overline{D}$, $\operatorname{osc}_{B(x,\delta_x)}u<\infty$. Let $x=(b_1, b_2, a) \in D$, $\widetilde{x}=(\widetilde{b}_1, \widetilde{b}_2, \widetilde{a}) \in D$ such that $\rho(x, \widetilde{x}) < \delta_x/32$, $|\mathbf{b}-\widetilde{\mathbf{b}}| \leqslant 1$ and $|a-\widetilde{a} + b_1\widetilde{b}_2 - b_2\widetilde{b}_1|\leqslant 1/2$. Consider the coupling from Theorem $\ref{thm:3}$ of two Brownian motions, $\textbf{X}$ and $\widetilde{\textbf{X}}$, on the Heisenberg group  starting from the points $x$ and $\tilde{x}$ respectively.

By Theorem \ref{thm:4} and the equivalence of the  Carnot-Carath\'{e}odory metric $d_{CC}$ and the pseudo-metric $\rho$, we have
\begin{align*}
\mathbb{P}\left(\tau>\tau_{B(x, \delta_x)}\left(\mathbf{X}\right)\wedge\widetilde{\tau}_{B(x, \delta_x)}\left(\widetilde{\mathbf{X}}\right)\right)\leqslant C\left(1+\frac{1}{\delta_x} + \frac{1}{\delta^4_x} + \frac{(1+\delta_x)^3}{\delta^4_x}\right)d_{CC}\left(x,\widetilde{x}\right).
\end{align*}
Using the coupling from Theorem $\ref{thm:3}$ and It\^{o}'s formula we
have that
\begin{align*}
\left|u\left(x\right)-u\left(\widetilde{x}\right)\right| & = \left|\mathbb{E}\left[u\left(\mathbf{X}_{\tau_{B(x, \delta_x)}\left(\mathbf{X}\right)}\right)-u\left(\widetilde{\mathbf{X}}_{\widetilde{\tau}_{B(x, \delta_x)}\left(\widetilde{\mathbf{X}}\right)}\right)\right]\right|\\
 & \leqslant  \mathbb{E}\left[\left|u\left(\mathbf{X}_{\tau_{B(x, \delta_x)}\left(\mathbf{X}\right)}\right)-u\left(\widetilde{\mathbf{X}}_{\widetilde{\tau}_{B(x, \delta_x)}\left(\widetilde{\mathbf{X}}\right)}\right)\right|\right]\\
 & \leqslant  \left(\operatorname{osc}_{B(x,\delta_x)}u\right) \cdot\mathbb{P}\left(\tau>\tau_{B(x, \delta_x)}\left(\mathbf{X}\right)\wedge\widetilde{\tau}_{B(x, \delta_x)}\left(\widetilde{\mathbf{X}}\right)\right)\\
 & \leqslant  C\left(\operatorname{osc}_{B(x,\delta_x)}u\right)\left(1+\frac{1}{\delta_x} + \frac{1}{\delta^4_x} + \frac{(1+\delta_x)^3}{\delta^4_x}\right)d_{CC}\left(x,\widetilde{x}\right).
\end{align*}
Since $u\in C\left(\overline{D}\right)\cap C^2\left(D\right)$ therefore (\ref{gradineq1})  holds for every $x\in D$. Dividing out by $d_{CC}\left(x,\widetilde{x}\right)$ and using (\ref{gradineq1}) we have that for every $x\in D$,
\begin{multline*}
\hspace{1cm} \left\Vert \nabla_{\mathcal{H}}u(x)\right\Vert _{\mathcal{H}} \leqslant \left|\nabla_{\mathcal{H}}u\right|(x)=\lim_{r\downarrow0}\sup_{0<d_{CC}\left(x,\tilde{x}\right)\leq r}\frac{\left|u\left(x\right)-u\left(\widetilde{x}\right)\right|}{d_{CC}\left(x,\widetilde{x}\right)}\\
 \leqslant C\left(1+\frac{1}{\delta_x} + \frac{1}{\delta^4_x} + \frac{(1+\delta_x)^3}{\delta^4_x}\right)\operatorname{osc}_{B(x,\delta_x)}u, \hspace{1cm}
\end{multline*}
as needed.
\end{proof}

\begin{corollary}\label{Cor4.5}
Let $u\in C\left(\overline{D}\right)\cap C^{\infty}\left(D\right)$ be a non-negative solution to
$\Delta_{\mathcal{H}}u=0$ on $D\subset\mathbb{H}^{3}$. There exists
a constant $C>0$ that does not depend on $u,\delta_{x},x,D$ such
that
\[
\left\Vert \nabla_{\mathcal{H}}u\left(x\right)\right\Vert _{\mathcal{H}}\leqslant \left|\nabla_{\mathcal{H}}u\right|(x)\leqslant C\left(1+\frac{1}{\delta_{x}}+\frac{1}{\delta_{x}^{4}}+\frac{\left(1+\delta_{x}\right)^{3}}{\delta_{x}^{4}}\right)u(x)
\]for every $x\in D$.
%The constant $C$ depends only on the choice of sub-Laplacian $\Delta_{\mathcal{H}}$
%and the pseudo-metric $\rho$ on $\mathbb{H}^{3}$.
\end{corollary}
\begin{proof}
By  \cite[Corollary 5.7.3]{BonfiglioliLanconelliUguzzoniBook} we have the following Harnack inequality
\begin{equation}
\sup_{B(x, \alpha^*\delta_{x})}u\leqslant C\inf_{B\left(x, \alpha^*\delta_{x}\right)}u \label{Harnack}
\end{equation}
for $x\in D\subset\mathbb{H}^{3}$, where $\alpha^* \in (0,1]$, $C>0$ are constants
% which depend only on the choice of sub-Laplacian $\Delta_{\mathcal{H}}$
%and the pseudo-metric $\rho$ but
not depending on $u,\delta_{x}, x, D$. Then Equations \eqref{gradineq2} and \eqref{Harnack} give the desired result.
\end{proof}

We can use Corollary \ref{Cor4.5} and the stratified structure of $\mathbb{H}^3$ to prove the Cheng-Yau gradient estimate. In particular, this recovers the fact that non-negative harmonic functions on the Heisenberg group must be constant.
%We point out that a combination of \cite[Proposition 2.5]{BaudoinBonnefont2016} and
%\cite[Theorem 1.2, Lemmma 2.3]{CoulhonJiangKoskelaSikora2017}
%can be used to obtain the Cheng-Yau inequality analytically for any Carnot group.
We thank F. Baudoin for pointing out the connection between the gradient estimate in Corollary \ref{Cor4.5} and the Cheng-Yau inequality.

\begin{corollary}\label{c.4.6}
If $u$ is any positive harmonic function in a ball $B\left(x_{0}, 2r\right)\subset\mathbb{H}^{3}$,
then there exists a universal constant $C>0$ not dependent on $u$ and $x_0$
such that
\[
\sup_{B(x_{0},r),}\left\Vert \nabla_{\mathcal{H}}\log u(x)\right\Vert _{\mathcal{H}}\leq\frac{C}{r}.
\]
Moreover, if $u$ is any positive harmonic function on $\mathbb{H}^{3}$, then $u$ must be a constant.
\end{corollary}

\begin{proof}

Suppose $u>0$ is harmonic in $B\left(0,2\right)$. By Corollary \ref{Cor4.5}

\begin{equation}
\frac{\left\Vert \nabla_{\mathcal{H}}u(x)\right\Vert _{\mathcal{H}}}{u(x)}\leq C^{\prime}=C\sup_{x\in B\left(0,1\right)}\left(1+\frac{1}{\delta_{x}}+\frac{1}{\delta_{x}}+\frac{\left(1+\delta_{x}\right)^{3}}{\delta_{x}^{4}}\right),\,\,\,\,x\in B(0,1),\label{eq:Cheng-Yau-1}
\end{equation}
where $C$ is the same constant as in Corollary \ref{Cor4.5}. This implies that
\begin{equation}
\sup_{B(0,1),}\left\Vert \nabla_{\mathcal{H}}\log u\right\Vert _{\mathcal{H}}\leqslant C^{\prime}.\label{eq:Cheng-Yau-2}
\end{equation}

Now suppose that $u>0$ is harmonic in $B\left(x_{0}, 2r\right)$
for $r>0$. By left invariance and the dilation properties of $\mathbb{H}^3$  we see that \eqref{eq:Cheng-Yau-2} implies
\[
\sup_{B(x_{0},r),}\left\Vert \nabla_{\mathcal{H}}\log u\right\Vert _{\mathcal{H}}\leqslant \frac{C^{\prime}}{r}.
\]
If $u$ is harmonic on all of $\mathbb{H}^{3}$, taking $r\to\infty$ gives us that $u$ must be constant.
\end{proof}

\section{Concluding remarks}
Our work gives the first use of explicit non-Markovian coupling techniques to get geometric information in the sub-Riemannian setting. We would like to point out some potentially significant connections with a different approach to such a setting. K.~Kuwada in \cite{Kuwada2010a} proved an important result on the duality of $L^q$-gradient estimates for the heat kernel of diffusions and their $L^p$-Wasserstein distances under the assumptions of volume doubling and a local Poincar\'{e} inequality, for any $p \in [1, \infty]$, $\frac{1}{p} + \frac{1}{q}=1$. Using this duality, he used the $L^1$-gradient estimate of the heat kernel for Brownian motion on the Heisenberg group obtained in \cite{LiHong-Quan2006} and \cite{BakryBaudoinBonnefontChafai2008} to derive $L^{\infty}$-Wasserstein bounds. More precisely, he proved that if $d_W(x,y;t)$ denotes the $L^{\infty}$-Wasserstein distance between the laws of Brownian motion on $\mathbb{H}^3$ starting from $x$ and $y$ at time $t>0$, then
\begin{equation}\label{equation:kuwada}
d_W(x,y;t) \leqslant K d_{CC}(x,y)
\end{equation}
for some constant $K$ that does not depend on $x,y,t$. The constant $K$ is not known, the best estimate obtained so far is $K \geqslant \sqrt{2}$ (see \cite{DriverMelcher2005}). Although we work with the total variation distance instead of the Wasserstein distance, Theorem \ref{thm:TVD} gives a better estimate of the distance between the laws of the two Brownian motions on $\mathbb{H}^3$, as it not only captures the dependence on the starting points, but also gives the ``\textit{polynomial decay}" in time.

%However, as described in of \cite[Example 4.3]{Kuwada2010a}, the heat kernel gradient bounds in %\cite{LiHong-Quan2006, BakryBaudoinBonnefontChafai2008}, via Kuwada's duality described above, %imply that \textit{for any $t>0$}, there exists a coupling $(\mathbf{X}, \widetilde{\mathbf{X}})$ %of two Brownian motions on $\mathbb{H}^3$ starting from $x$ and $y$ respectively such that
%\[
%d_{CC}(\mathbf{X}_t, \widetilde{\mathbf{X}}_t) \leqslant  K d_{CC}(x,y)
%\]
%where $K$ is the same constant as in \eqref{equation:kuwada}. The remarkable power of Kuwada's %duality is that if we could construct the above coupling directly, it would give a fairly %elementary proof of the gradient estimates in \cite{LiHong-Quan2006} and %\cite{BakryBaudoinBonnefontChafai2008} (which were obtained using heavy analytic tools) via %\textit{purely probabilistic methods}. However, this coupling has been so far out of reach as, %quoting Kuwada in \cite{Kuwada2010a}, ``it is sometimes a complicated issue to construct %well-behaved couplings in the absence of curvature bounds". One important reason why the coupling %has been hard to construct so far is that mostly Markovian couplings were sought after (see %\cite{Kendall2007a}) as non-Markovian couplings were typically considered implicit.

Our intention is to use the techniques developed in this article and in \cite{BanerjeeKendall2016a}, to give a systematic way to explicitly construct non-Markovian couplings via spectral expansions, and connect it to the previous results on the heat kernels such as those in \cite{LiHong-Quan2006, DriverMelcher2005, Kuwada2010a}.
%Moreover, as these couplings are typically explicit, one can produce better estimates for the %constant $K$.
This will be addressed in future work.
\begin{acknowledgement}
S.B. is grateful for many helpful and motivating conversations with W. S. Kendall.
The authors also thank Fabrice Baudoin for drawing our attention to Kuwada's results and  the Cheng-Yau estimate in Corollary \ref{c.4.6}, and Iddo Ben-Ari for his interest in the subject and his insight into coupling techniques. We also thank the anonymous referee whose careful review and suggestions greatly improved the presentation of this paper.
\end{acknowledgement}

\bibliographystyle{plain}	% (uses file "plain.bst")

\end{document}